\documentclass[final]{siamart250211}

\usepackage{amsmath,amssymb,geometry}
\newtheorem{remark}{Remark}

\usepackage{siunitx}
\usepackage{tikz}
\usetikzlibrary{calc,decorations.pathmorphing,decorations.pathreplacing}
\usepackage{subcaption}

\title{Reliable Topology for Dynamic Data:\\Mathematical Foundations and Applications}
\author{Chad M. Topaz\thanks{Williams College, Williamstown, MA; and University of Colorado, Boulder, CO (cmt6@williams.edu).}}

\headers{Reliable Topology for Dynamic Data}{C.M. Topaz}

\begin{document}

\maketitle

% Abstract, keywords, and MSC codes
\begin{abstract}
Across many scientific domains, practitioners rely on coarse, discretized summaries to track the evolving structure of complex systems under noise, measurement error, and changing system size. Understanding when such summaries are reliable---and when apparent robustness is illusory---remains a fundamental challenge. Topological data analysis (TDA) provides a case study: Crocker diagrams track the number of topological features across spatial scale and time, and because they are computationally efficient and easy to interpret, they have been widely used for exploratory analysis, bifurcation detection, model selection, and parameter inference. Despite their popularity, Crocker diagrams have lacked rigorous stability guarantees ensuring robustness to small data distortions. We develop a conditional stability theory for Crocker diagrams constructed from evolving point clouds. Our main results include deterministic conditions guaranteeing exact invariance when pairwise distances are well separated from the diagram's discretization thresholds, together with bounds on how much the diagrams can change when these conditions fail. We also establish probabilistic stability guarantees under Gaussian noise and bounds on topological change caused by adding or removing points, scaling linearly with the number of modified points. We illustrate these results using two complementary examples: an analytically tractable breathing polygon model that reveals how stability thresholds depend on geometry, and a feasibility analysis of epithelial cell imaging data showing when bounded-change guarantees provide the appropriate robustness framework. Together, these results reveal a two-tier stability structure for coarse, discretized topological summaries: exact invariance under verifiable geometric separation conditions, and geometry-controlled bounded change otherwise.
\end{abstract}


\section{Introduction}

Data with complex geometric structure is now ubiquitous across the sciences, arising in contexts ranging from protein-folding trajectories~\cite{XiaWei2014} and neuronal firing patterns~\cite{Giusti2015} to social networks~\cite{Aktas2019} and atmospheric vortices~\cite{Dorrington2025}. Extracting meaningful insights from such data requires methods that capture essential shape and organization beyond traditional statistical summaries. In dynamic settings, this challenge is amplified: as structures evolve over time, summaries must remain interpretable despite continuous change.

A recurring tension in applied mathematics is that the summaries most useful in practice---coarse, discretized, and computationally efficient---are often those least amenable to classical continuity-based stability analysis. Yet such summaries frequently appear robust in empirical settings. Understanding when this apparent robustness is genuine, and when it is illusory, is a foundational reliability question for dynamic data analysis.

Topological data analysis (TDA) offers a principled, algebraic framework for quantifying data shape through \emph{persistent homology}, which detects features such as connected components, loops, and voids across a continuously varying scale parameter. These features are typically summarized by \emph{persistence diagrams}, which record the birth and death scales of each feature and have revealed hidden structure in neuroscience~\cite{stolz2021topological}, climate science~\cite{Tymochko2020}, finance~\cite{GideaKatz2018}, materials science~\cite{Hiraoka2016}, and medical imaging~\cite{Somasundaram2021}, to name a few areas.

A key theoretical advantage of persistence diagrams is their stability: small perturbations in the input data induce proportionally small changes in the resulting diagrams. The foundational stability theorem of~\cite{CohenSteiner2007} formalizes this principle, bounding diagram perturbations by the size of the underlying data perturbation and providing a rigorous foundation for robust topological inference.

These guarantees have motivated a variety of vector-valued summaries that inherit stability while facilitating statistical analysis and machine learning. \emph{Persistence landscapes}~\cite{Bubenik2015} and \emph{persistence images}~\cite{Adams2017} provide stable finite-dimensional representations that enable downstream quantitative tasks.

For time-varying data, \emph{vineyards} track the continuous evolution of individual persistence-diagram features and enjoy their own stability guarantees~\cite{CohenSteiner2006}. However, vineyard computations scale poorly with data size and temporal resolution. In contrast, simpler summaries such as Betti curves---recording only feature counts---are computationally efficient but theoretically fragile: arbitrarily small perturbations can instantaneously change Betti numbers in worst-case constructions. This tension raises a natural question: why do such coarse summaries often appear robust in practice?

Despite this theoretical vulnerability, coarse summaries remain popular for large, noisy datasets due to their computational efficiency---often orders of magnitude faster than vineyard analysis---and their interpretability. \emph{Crocker diagrams} extend Betti curves to dynamic settings by stacking feature counts over both scale and time, producing two-dimensional heat maps. They serve as a representative example of discretized, count-based summaries that trade fine geometric detail for interpretability and computational efficiency. Similar discretization--threshold effects arise in many count-based summaries of evolving systems. This efficiency has enabled applications to biological swarms~\cite{Ulmer2019}, mesenchymal cell migration~\cite{Nguyen2024}, and bifurcation detection~\cite{Guzel2022}. Empirically, Crocker diagrams often appear robust, yet until now no rigorous stability theory existed to explain this behavior.

A key distinction between Crocker diagrams and classical persistence diagrams is that Crocker diagrams are computed on a fixed, discrete grid of scale values. As we show in this paper, the placement of this grid relative to the collection of pairwise distances---which determines the critical distances at which topological changes occur---plays a central role in stability. When grid values are well separated from these critical distances by a quantifiable margin we call \emph{grid-threshold clearance}, Crocker diagrams can be exactly invariant under perturbations. Crucially, this separation condition is directly computable from the data and the chosen discretization, making it a practical diagnostic rather than a purely asymptotic assumption. In contrast, near-threshold configurations, where many distances cluster around grid values, are inherently less stable. This clearance-based perspective resolves the apparent paradox between empirical robustness and worst-case instability. Our results do not assert that Crocker diagrams are generically more stable than persistence diagrams; rather, they explain why, in discretized settings with sufficient separation, Crocker diagrams can exhibit robust behavior under explicit, checkable conditions. This theoretical gap is particularly consequential for experimental applications, where measurement uncertainty and changing point sets are unavoidable. In biological imaging, for example, localization error affects point positions, while cell division, death, or migration introduce point insertions and deletions over time.

We address this gap in the literature by developing the first comprehensive stability theory for Crocker diagrams. Because these diagrams aggregate discrete feature counts, they lack the continuity underpinning classical stability theorems. Our analysis therefore treats deterministic perturbations, random Gaussian noise, and point insertion/deletion scenarios directly. Through an analytically tractable breathing polygon model and a feasibility analysis for epithelial cell imaging, we demonstrate how theoretical bounds translate into practical diagnostics---distinguishing settings where exact stability can be certified from those where bounded-change guarantees provide the appropriate robustness framework.

A central message of this work is that Crocker-diagram stability is inherently two-tiered: exact invariance is attainable for structured configurations with verifiable geometric separation, while bounded-change guarantees explain robustness in large, noisy, or dynamically evolving systems where such separation cannot be verified.

We derive rigorous, interpretable bounds that clarify when and how Crocker diagrams constructed from point clouds in $\mathbb{R}^d$ remain stable under realistic modeling assumptions:
\begin{itemize}
  \item \textbf{Deterministic perturbations:} exact invariance when point displacements are smaller than half the grid-threshold clearance (the minimum separation between grid thresholds and critical distances), and geometry-controlled bounded change otherwise;
  \item \textbf{Random (Gaussian) noise:} high-probability stability when this separation dominates the noise scale, with explicit concentration bounds accounting for the number of points, scales, and time frames;
  \item \textbf{Point insertions/deletions:} Betti-count changes that grow at most linearly with the number of modified points, controlled by local density rather than global system size.
\end{itemize}

Together, these results place a widely used class of discretized topological summaries on a firm theoretical foundation, complementing their practical utility with a principled understanding of robustness across diverse application domains.

The remainder of the paper is organized as follows. Section~\ref{sec:background} provides background on TDA and Crocker diagrams. Section~\ref{sec:deterministic} establishes deterministic stability bounds, and Section~\ref{sec:breathing-poly} illustrates them using the breathing polygon model. Section~\ref{sec:probabilistic} extends the theory to Gaussian noise, while Section~\ref{sec:point-churn} addresses point insertions and deletions. We conclude with practical recommendations for reliable Crocker-based analysis of evolving systems.

\section{Background on Topological Data Analysis and Crocker Diagrams}
\label{sec:background}

Topological data analysis (TDA) offers a suite of tools to summarize the shape of data. A hallmark feature of many such tools is the stability of their outputs under small perturbations to the input. For example, the classical stability theorem for persistence diagrams states that if two finite point clouds $P$ and $\tilde{P}$ in $\mathbb{R}^d$ are close in Hausdorff distance, then the bottleneck distance between their persistence diagrams $D(P)$ and $D(\tilde{P})$ is also small:
\begin{equation}
\operatorname{Bott}(D(P), D(\tilde{P})) \le 2 \, \operatorname{Haus}(P, \tilde{P}).
\end{equation}
This inequality implies that topological features may shift slightly in their birth and death scales, but cannot appear or disappear abruptly unless they were already short-lived. The underlying reason is that persistent homology is a Lipschitz-continuous functor from filtered simplicial complexes to persistence diagrams.

However, Crocker diagrams are not persistence diagrams. While persistence diagrams track the birth and death of individual topological features across all scales, Crocker diagrams record the aggregate count of topological features---specifically, the Betti number $\beta_k(\varepsilon, t)$---on a rectangular grid of scale-time values. Crocker diagrams do not preserve information about correspondence between features; they record only the total count at each grid point. Because Betti numbers are integer-valued and piecewise-constant, even a small perturbation can cause a jump in value. Consequently, the function $(\varepsilon, t) \mapsto \beta_k(\varepsilon, t)$ is not continuous, and classical stability results do not directly apply. Despite this theoretical challenge, Crocker diagrams often perform well empirically.

To analyze the stability of Crocker diagrams, we must first establish our framework. Let $P(t_i) = \{p_1(t_i), \dots, p_m(t_i)\}$ be a set of $m$ points in $\mathbb{R}^d$ observed at discrete times $t_i$. A Crocker diagram tabulates Betti numbers on a discrete grid of scales and times, giving entries $\beta_k(\widehat{\varepsilon}_j,t_i)$ for $(\widehat{\varepsilon}_j,t_i)\in\{\widehat{\varepsilon}_1,\dots,\widehat{\varepsilon}_{n_\varepsilon}\}\times\{t_1,\dots,t_{n_t}\}$, where $\{\widehat{\varepsilon}_j\}_{j=1}^{n_\varepsilon}$ represents our fixed scale grid for recording Betti numbers. We measure differences between Crocker diagrams using the entrywise $\ell_1$ norm, that is, the sum of absolute differences across all grid cells.

For each scale $\varepsilon \ge 0$, we construct the Vietoris--Rips complex $\operatorname{VR}\bigl(P(t_i),\varepsilon\bigr)$. A \emph{flag complex} is a simplicial complex in which a simplex is included whenever all of its edges are present. The Vietoris--Rips complex is a flag complex in which an edge between two vertices appears precisely when their distance is at most $\varepsilon$; higher-dimensional simplices are then included automatically whenever all their edges are present. The $k$-th Betti number $\beta_k(\varepsilon,t_i)$ counts $k$-dimensional topological features such as connected components ($k=0$), loops ($k=1$), and voids ($k=2$).

To analyze how perturbations of the point cloud affect this Crocker diagram, consider a perturbed configuration $\tilde{P}(t_i)=\{\tilde{p}_1(t_i),\dots,\tilde{p}_m(t_i)\}$ where each point moves by at most $\delta$:
\begin{equation}
\max_{i}\max_{a}\|p_a(t_i)-\tilde{p}_a(t_i)\|\le\delta.
\end{equation}
Here $\delta > 0$ denotes the maximum perturbation magnitude---the largest displacement of any point across all times.

At each fixed time $t_i$, the topology can change only when the scale parameter $\varepsilon$ crosses certain critical values. These critical scales are determined by the collection of pairwise distances among the points: they are the values at which edges appear in the Vietoris--Rips filtration. In typical datasets, some pairs of points may have identical distances, especially in structured configurations like grids or symmetric shapes. Thus, we define the critical scales as the sorted \emph{distinct} pairwise distances at each sampled time $t_i$:
\begin{equation}
0<\varepsilon_1(t_i)<\varepsilon_2(t_i)<\dots<\varepsilon_{M_i}(t_i),
\end{equation}
where each $\varepsilon_r(t_i)$ represents a distinct distance value observed in the data at time $t_i$.

Define the minimal spacing between consecutive critical scales across all sampled times by
\begin{equation}
\Delta = \min_{1\le i\le n_t}\min_{1\le r < M_i}\{\varepsilon_{r+1}(t_i)-\varepsilon_r(t_i)\},
\end{equation}
where $M_i$ is the number of distinct critical distances at time $t_i$.

For our exact-stability results, we fix a Crocker scale grid $\{\widehat{\varepsilon}_j\}_{j=1}^{n_\varepsilon}$ and assume that every grid value $\widehat{\varepsilon}_j$ lies strictly between two consecutive critical distances at \emph{every} sampled time. Formally, for each time $t_i$ and each grid index $j$ there exists $r = r(i,j)$ such that
\begin{equation}
\varepsilon_{r(i,j)}(t_i) < \widehat{\varepsilon}_j < \varepsilon_{r(i,j)+1}(t_i).
\end{equation}
We call this the \emph{in-gap condition}. For bookkeeping convenience, we formally extend the critical distance list to include boundary values: we set $\varepsilon_0(t_i) = 0$ and $\varepsilon_{M_i+1}(t_i) = +\infty$ (or a value exceeding all grid scales); these do not correspond to actual pairwise distances. This extension allows grid values below the smallest positive pairwise distance or above the largest to satisfy the in-gap condition, as is common in practice when probing the $\beta_0 = m$ regime at small scales or the fully-connected regime at large scales.

Because the critical distances $\varepsilon_r(t)$ evolve with time, the index $r(i,j)$ may change from one frame to the next. The in-gap condition is indispensable for exact stability: if a grid value ever coincides with a critical distance, exact stability may fail even when point motions remain small. In typical applications, this constraint is mild because scale grids are usually chosen to probe structure \emph{between} critical transitions rather than at them.

When the in-gap condition fails---for example, when two rows of points are separated by a distance exactly equal to a grid scale---arbitrarily small perturbations can produce extensive topological change. Our exact-stability theory necessarily excludes such degenerate configurations; the bounded-change results we develop later apply more broadly.

\section{Stability to Point Cloud Perturbations}
\label{sec:deterministic}

In this section we develop a complete stability framework for Crocker diagrams subjected to perturbations of the underlying point cloud. Three data-dependent quantities anchor the theory:
\begin{itemize}
  \item the \emph{minimum critical-distance gap} $\Delta$, measuring the spacing between consecutive distinct pairwise distances in the point cloud;
  \item the grid clearance $\Gamma$, defined as the minimum distance between any sampled scale $\varepsilon_j$ and the nearest critical distance; and
  \item a \emph{local-density bound} $\Lambda_\delta(\varepsilon)$ that caps how many points can lie within a given scale $\varepsilon$ of any point, accounting for perturbations of magnitude $\delta$.
\end{itemize}
We first prove an exact-stability criterion: if the perturbation amplitude satisfies $\delta < \Gamma/2$, every Betti number in the diagram remains unchanged. When this condition fails, we derive a combinatorial bound showing that the total change in Betti numbers is limited by $\Lambda_\delta$, the number of perturbed points, and the size of the Crocker grid. As we will show, these results turn familiar experimental or numerical levers into interpretable inequalities. We illustrate the theory with two examples: (i)~an analytic ``breathing-polygon'' model where the clearance can be computed explicitly, and (ii)~migrating epithelial cell sheets, demonstrating how the bounded-change theorem guides experimental design when exact-stability conditions may be difficult to verify.

\subsection{Theoretical Stability Results}
\label{sec:stability_theory}

We begin by analyzing the behavior of the Vietoris--Rips complex under perturbations. Our analysis proceeds in two stages: first establishing when Crocker diagrams remain exactly unchanged despite perturbations, and then bounding the potential changes when exact stability no longer holds.

Consider a point cloud $P(t_i) = \{p_1(t_i), \dots, p_m(t_i)\}$ observed at discrete times $t_i$. Let its perturbed version be $\tilde{P}(t_i) = \{\tilde{p}_1(t_i),\dots,\tilde{p}_m(t_i)\}$ with a bounded perturbation:
\begin{equation}
\max_{i}\max_{a}\|p_a(t_i)-\tilde{p}_a(t_i)\|\le\delta.
\end{equation}

For any pair of points $(a,b)$ at a fixed time $t_i$, the perturbation shifts their pairwise distance by at most $2\delta$, as guaranteed by the triangle inequality:
\begin{equation}
\bigl|\|\tilde{p}_a(t_i)-\tilde{p}_b(t_i)\|-\|p_a(t_i)-p_b(t_i)\|\bigr|\le 2\delta.
\label{eq:distance_shift}
\end{equation}

To state our stability criterion precisely, we must quantify how close the Crocker grid values $\widehat{\varepsilon}_j$ lie to the critical distances. Recall from Section~\ref{sec:background} that the critical distances at time $t_i$ are the sorted distinct pairwise distances $\varepsilon_1(t_i) < \varepsilon_2(t_i) < \cdots < \varepsilon_{M_i}(t_i)$, and that the in-gap condition requires each grid value $\widehat{\varepsilon}_j$ to satisfy $\varepsilon_{r(i,j)}(t_i) < \widehat{\varepsilon}_j < \varepsilon_{r(i,j)+1}(t_i)$ for some index $r(i,j)$.

\begin{definition}[Grid Clearance]
\label{def:clearance}
For each sampled time $t_i$ and grid scale $\widehat{\varepsilon}_j$, the \emph{cell clearance} is
\begin{equation}
g_{i,j} := \min\bigl(\widehat{\varepsilon}_j - \varepsilon_{r(i,j)}(t_i),\; \varepsilon_{r(i,j)+1}(t_i) - \widehat{\varepsilon}_j\bigr),
\end{equation}
measuring the distance from $\widehat{\varepsilon}_j$ to the nearest critical distance at time $t_i$. The \emph{global grid clearance} is
\begin{equation}
\Gamma := \min_{i,j} g_{i,j},
\end{equation}
the minimum clearance across all grid cells.
\end{definition}

The in-gap condition guarantees $g_{i,j} > 0$ for all $(i,j)$, hence $\Gamma > 0$. However, the in-gap condition alone does not control how small $\Gamma$ can be. The grid clearance $\Gamma$ depends on both the critical-distance spectrum and the placement of the grid values $\widehat{\varepsilon}_j$ relative to that spectrum.

\begin{theorem}[Exact Stability]
\label{thm:exact_stability}
Suppose the in-gap condition holds, and let $\Gamma$ be the global grid clearance. If
\begin{equation}
\delta < \frac{\Gamma}{2},
\label{eq:delta_bound}
\end{equation}
then the Crocker diagram remains exactly stable:
\begin{equation}
\beta_k(\tilde{P}(t_i),\widehat{\varepsilon}_j)=\beta_k(P(t_i),\widehat{\varepsilon}_j) \quad\text{for all }k,j,i.
\end{equation}
\end{theorem}

\begin{proof}
An edge between points $p_a(t_i)$ and $p_b(t_i)$ exists in the Vietoris--Rips complex at scale $\widehat{\varepsilon}_j$ if and only if $\|p_a(t_i) - p_b(t_i)\| \le \widehat{\varepsilon}_j$. For the Crocker diagram to remain unchanged, we need each edge's status (present or absent) to be preserved under perturbation.

By~\eqref{eq:distance_shift}, each pairwise distance changes by at most $2\delta$. For the edge status at $\widehat{\varepsilon}_j$ to change, the perturbed distance must cross the threshold $\widehat{\varepsilon}_j$. This can only happen if the original distance lies within $2\delta$ of $\widehat{\varepsilon}_j$.

Since the critical distances are exactly the sorted distinct pairwise distances, and the in-gap condition places $\widehat{\varepsilon}_j$ strictly between consecutive critical distances $\varepsilon_{r(i,j)}(t_i)$ and $\varepsilon_{r(i,j)+1}(t_i)$, there are no pairwise distances strictly between these bracketing values. Therefore, the nearest pairwise distance to $\widehat{\varepsilon}_j$ is at distance at least $g_{i,j}$:
\begin{equation}
\min_{a<b} \bigl| \|p_a(t_i) - p_b(t_i)\| - \widehat{\varepsilon}_j \bigr| \geq g_{i,j} \geq \Gamma.
\end{equation}
Therefore, if $2\delta < \Gamma$, no pairwise distance can cross any grid threshold $\widehat{\varepsilon}_j$.

Since all edge statuses are preserved at every grid scale and time, and the Vietoris--Rips complex is a flag complex determined by its edges, the entire simplicial complex---and consequently all Betti numbers---remain unchanged.
\end{proof}

\begin{remark}
\label{rem:gridclearance}
The grid clearance $\Gamma$ is related to, but distinct from, the minimum critical-distance gap $\Delta$. At each cell $(i,j)$, the grid value $\widehat{\varepsilon}_j$ lies in some gap $(\varepsilon_r(t_i), \varepsilon_{r+1}(t_i))$, and the cell clearance satisfies
\begin{equation}
g_{i,j} \le \frac{\varepsilon_{r+1}(t_i) - \varepsilon_r(t_i)}{2}.
\end{equation}
To relate $\Gamma$ to a gap-based quantity, define the \emph{grid-relevant gap}
\begin{equation}
\Delta_{\mathrm{grid}} := \min_{i,j} \bigl(\varepsilon_{r(i,j)+1}(t_i) - \varepsilon_{r(i,j)}(t_i)\bigr),
\end{equation}
the minimum width among the specific gaps that contain some grid value at some sampled time. Then $\Gamma \le \Delta_{\mathrm{grid}}/2$. Since $\Delta_{\mathrm{grid}} \ge \Delta$ (because $\Delta$ is a minimum over all gaps while $\Delta_{\mathrm{grid}}$ is a minimum over a subset), there is no universal inequality of the form $\Gamma \le \Delta/2$. In particular, $\Gamma$ can be much larger than $\Delta/2$ if the grid never lands in the narrowest gaps. In practice, one should compute $\Gamma$ directly from the data and grid specification rather than relying on gap-based estimates.
\end{remark}

This exact-stability theorem provides a fundamental guarantee: if measurement errors are sufficiently small compared to the grid clearance, the topological summary remains completely unaffected. The result is most useful when the distance spectrum has well-separated families (as in symmetric configurations or lattices), so that $\Gamma$ is not too small.

What happens when $\delta \geq \Gamma/2$ and we can no longer guarantee exact stability? In this case, some pairwise distances may cross grid thresholds, potentially altering the topology. However, these changes are localized to simplices containing perturbed vertices. To quantify these changes, we introduce a parameter measuring the local density of points.

\begin{definition}[Local Density Parameter]
\label{def:lambda}
For each scale $\widehat{\varepsilon}_j$ and perturbation bound $\delta \ge 0$, the \emph{$\delta$-inflated local density parameter} $\Lambda_\delta(\widehat{\varepsilon}_j)$ measures the maximum number of points within distance $\widehat{\varepsilon}_j + 2\delta$ of any single point, across all points and time steps:
\begin{equation}
\Lambda_\delta(\widehat{\varepsilon}_j) = \max_{i=1,\dots,n_t} \max_{a=1,\dots,m} \# \left\{ \ell : \|p_\ell(t_i) - p_a(t_i)\| \le \widehat{\varepsilon}_j + 2\delta \right\}.
\end{equation}
When $\delta = 0$, we write $\Lambda(\widehat{\varepsilon}_j) := \Lambda_0(\widehat{\varepsilon}_j)$ for the \emph{unperturbed} local density.
\end{definition}

This count always includes the point $p_a$ itself (since $\|p_a-p_a\|=0 \leq \widehat{\varepsilon}_j + 2\delta$), so $\Lambda_\delta(\widehat{\varepsilon}_j) \geq 1$ for any $\widehat{\varepsilon}_j > 0$. The inflation by $2\delta$ accounts for the fact that when points move by at most $\delta$, two points at perturbed distance $\widehat{\varepsilon}_j$ were originally at distance at most $\widehat{\varepsilon}_j + 2\delta$.

The parameter $\Lambda_\delta(\widehat{\varepsilon}_j)$ provides crucial insight into the intrinsic dimensionality and density of the data. For example, points sampled from a 1-dimensional curve in $\mathbb{R}^d$ will have $\Lambda_\delta(\widehat{\varepsilon}_j) \approx 2 + C\,(\widehat{\varepsilon}_j + 2\delta)$ for some constant $C$, reflecting the fact that each point has approximately two neighbors along the curve. Similarly, points sampled from a 2-dimensional surface will have $\Lambda_\delta(\widehat{\varepsilon}_j) \approx \pi\,(\widehat{\varepsilon}_j + 2\delta)^2\,\rho$, where $\rho$ is the sampling density. In general, lower-dimensional manifolds embedded in $\mathbb{R}^d$ will have smaller $\Lambda_\delta(\widehat{\varepsilon}_j)$ values than uniformly distributed points in the full ambient space.

To bound the potential changes in topology, we first need to understand how many simplices can be affected by perturbing a single point:

\begin{lemma}[Simplex Participation Bound]
\label{lem:simplex_bound}
In the standard convention where a $k$-simplex has $k+1$ vertices, a vertex in the \emph{perturbed} Vietoris--Rips complex at scale $\widehat{\varepsilon}_j$ participates in at most $\binom{\Lambda_\delta(\widehat{\varepsilon}_j)-1}{k}$ $k$-simplices and at most $\binom{\Lambda_\delta(\widehat{\varepsilon}_j)-1}{k+1}$ $(k+1)$-simplices.
\end{lemma}

\begin{proof}
Consider a vertex $\tilde{p}_a(t_i)$ in the perturbed configuration. Any neighbor $\tilde{p}_b(t_i)$ at scale $\widehat{\varepsilon}_j$ satisfies $\|\tilde{p}_a(t_i) - \tilde{p}_b(t_i)\| \le \widehat{\varepsilon}_j$. By the triangle inequality and the perturbation bound,
\begin{equation}
\|p_a(t_i) - p_b(t_i)\| \le \|\tilde{p}_a(t_i) - \tilde{p}_b(t_i)\| + 2\delta \le \widehat{\varepsilon}_j + 2\delta.
\end{equation}
Thus, every neighbor of $\tilde{p}_a$ in the perturbed complex at scale $\widehat{\varepsilon}_j$ corresponds to a point within distance $\widehat{\varepsilon}_j + 2\delta$ of $p_a$ in the original configuration. There are at most $\Lambda_\delta(\widehat{\varepsilon}_j) - 1$ such points (excluding $p_a$ itself). To form a $k$-simplex containing our vertex, we select $k$ of these neighbors, yielding at most $\binom{\Lambda_\delta(\widehat{\varepsilon}_j)-1}{k}$ possible $k$-simplices. Similarly, $(k+1)$-simplices require selecting $k+1$ neighbors, giving $\binom{\Lambda_\delta(\widehat{\varepsilon}_j)-1}{k+1}$ possibilities.
\end{proof}

To translate these simplex counts into changes in Betti numbers, recall that the $k$th Betti number is computed as:
\begin{equation}
\beta_k = \dim \ker \partial_k - \dim \operatorname{im} \partial_{k+1},
\end{equation}
where $\partial_k$ is the $k$th boundary map in the simplicial chain complex. When a point is perturbed, changes to $\beta_k$ can arise from:
\begin{enumerate}
\item[(i)] $k$-simplices added or removed, affecting $\ker \partial_k$; or
\item[(ii)] $(k+1)$-simplices added or removed, affecting $\operatorname{im} \partial_{k+1}$.
\end{enumerate}

\begin{proposition}[Change in Betti Number per Point]
\label{prop:betti_change}
For perturbations bounded by $\delta$, the maximum possible change in $\beta_k$ due to a single perturbed point is bounded by:
\begin{equation}
|\Delta \beta_k| \leq \binom{\Lambda_\delta(\widehat{\varepsilon}_j)-1}{k} + \binom{\Lambda_\delta(\widehat{\varepsilon}_j)-1}{k+1} = \binom{\Lambda_\delta(\widehat{\varepsilon}_j)}{k+1}.
\end{equation}
\end{proposition}

\begin{proof}
Adding or removing a single simplex corresponds to adding or removing one column in the relevant boundary matrix: $\partial_k$ for a $k$-simplex, $\partial_{k+1}$ for a $(k+1)$-simplex. A rank-one column update changes the matrix rank by at most~1, so the associated kernel or image dimension changes by at most~1. By Lemma~\ref{lem:simplex_bound}, the number of $k$-simplices and $(k+1)$-simplices containing a given vertex is bounded by $\binom{\Lambda_\delta-1}{k}$ and $\binom{\Lambda_\delta-1}{k+1}$ respectively. Since each simplex change alters $\beta_k$ by at most one, the total potential change is bounded by their sum. Pascal's identity $\binom{n-1}{k} + \binom{n-1}{k+1} = \binom{n}{k+1}$ yields the simplified form.
\end{proof}

\begin{theorem}[Global Stability Bound]
\label{thm:global_bound}
Let $B_k^{\text{orig}}$ and $B_k^{\text{mod}}$ denote the matrices containing the $k$-th Betti numbers for the original and perturbed point clouds, respectively. Let $m^* \leq m$ be the number of perturbed vertices, and suppose all perturbations are bounded by $\delta$. The cumulative topological change across all entries in the Crocker diagram is bounded by:
\begin{equation}
\label{eq:comb_bound}
\| B_k^{\text{mod}} - B_k^{\text{orig}} \|_1
\;\le\;
n_t \, m^* \,
\sum_{j=1}^{n_\varepsilon}
\binom{\Lambda_\delta(\widehat{\varepsilon}_j)}{k+1},
\end{equation}
where $\|\cdot\|_1$ denotes the entrywise $\ell_1$ norm.
\end{theorem}

\begin{proof}
By Proposition~\ref{prop:betti_change}, each perturbed vertex contributes at most $\binom{\Lambda_\delta(\widehat{\varepsilon}_j)}{k+1}$ to the change in $\beta_k$ at each grid point $(\widehat{\varepsilon}_j, t_i)$. With $m^*$ perturbed vertices across $n_t$ time steps, and summing over all $n_\varepsilon$ scale values, we obtain the bound.

This is a worst-case estimate: simplices with multiple perturbed vertices are counted separately for each vertex, leading to potential overcounting. However, the bound's dependence on local density $\Lambda_\delta$ rather than global point count $m$ makes it practically useful. Note that when $\delta$ is small, $\Lambda_\delta(\widehat{\varepsilon}_j) \approx \Lambda(\widehat{\varepsilon}_j)$, recovering the intuition that stability depends on local geometry.
\end{proof}

Taken together, these results provide a two-tiered stability framework: exact stability (Theorem~\ref{thm:exact_stability}) guarantees unchanged Betti numbers when perturbations are small relative to the grid clearance, while bounded stability (Theorem~\ref{thm:global_bound}) caps the total change when exact stability cannot be verified. The bounded-change theorem applies universally and depends only on local density, explaining why Crocker diagrams often perform well empirically even when the stringent exact-stability conditions are not met.

We now apply this theory to concrete examples that highlight the geometry-dependent nature of these bounds.

\subsection{Analytical Application: Breathing Polygons}
\label{sec:breathing-poly}

To illustrate our stability results, we consider a dynamic point cloud $P(t)$ consisting of particles in $\mathbb{R}^2$ forming a regular polygon whose size oscillates over time. This \emph{breathing polygon} provides a simple but instructive analytical case where the critical distances form well-separated families, allowing explicit computation of both the critical-distance gap $\Delta$ and the grid clearance $\Gamma$.

We define the circumscribed radius of the polygon to vary over time according to
\begin{equation}
a(t)=1+\frac{1}{2}\sin(t), \quad 0\le t<2\pi .
\end{equation}
At each time $t$, we place the $m$ vertices of the regular polygon on the circle of radius $a(t)$, using equal angular spacing. Specifically, the $v$-th vertex is positioned at
\begin{equation}
p_v(t)=\left(a(t)\cos\left(\frac{2\pi v}{m}\right),a(t)\sin\left(\frac{2\pi v}{m}\right)\right),
\quad v=0,1,\dots,m-1.
\end{equation}

The pairwise distances between vertices are chords of the circle. For vertices separated by $\ell$ steps, the chord length is
\begin{equation}
c_\ell(t)=2a(t)\sin\left(\frac{\pi\ell}{m}\right), 
\quad \ell=1,2,\ldots,\left\lfloor\frac{m}{2}\right\rfloor .
\end{equation}
Here $c_\ell(t)$ represents a distinct pairwise distance at time $t$, with $c_1(t)$ the side length and larger $\ell$ giving longer diagonals. Crucially, for a regular $m$-gon there are only $\lfloor m/2 \rfloor$ distinct distance families---far fewer than the $\binom{m}{2}$ total pairs---making this an ideal setting for exact-stability analysis.

To characterize the critical-distance spectrum, we examine consecutive chord lengths:
\begin{equation}
\Delta_\ell(t)=c_{\ell+1}(t)-c_\ell(t)
              =2a(t)\bigl[\sin\bigl((\ell+1)\theta\bigr)-\sin\bigl(\ell\theta\bigr)\bigr],
\qquad \theta=\frac{\pi}{m}.
\end{equation}
Using $\sin A-\sin B = 2\sin\bigl(\frac{A-B}{2}\bigr)\cos\bigl(\frac{A+B}{2}\bigr)$, this becomes
\begin{equation}
\Delta_\ell(t)=4a(t)\sin(\theta/2)\cos(\ell\theta+\theta/2).
\end{equation}
Since $\cos x$ is strictly decreasing on $(0,\pi)$, the smallest gap occurs at:
\begin{equation}
\ell^{\star}=
\begin{cases}
\lfloor m/2\rfloor-1 & \text{for } m\ge 5\\
1 & \text{for } m=4\\
\text{(only one distance)} & \text{for } m=3.
\end{cases}
\end{equation}

Because $\Delta_\ell(t)$ is proportional to $a(t)$, it attains its minimum when $a(t)$ is smallest, namely at $t=\frac{3\pi}{2}$ with $a_{\min}=1/2$. Hence the minimum critical-distance gap is
\begin{equation}
\Delta=2a_{\min}\bigl[\sin\bigl((\ell^{\star}+1)\theta\bigr)-\sin\bigl(\ell^{\star}\theta\bigr)\bigr].
\label{eq:global-delta}
\end{equation}
Table~\ref{tab:stability} provides the numerical values for small $m$.

\begin{table}[h]
\caption{Critical-distance gaps for regular $m$-gons with breathing radius $a(t)=1+\frac{1}{2}\sin t$. The minimum gap $\Delta$ characterizes the distance spectrum. The grid clearance $\Gamma$ depends on grid placement and satisfies $\Gamma \le \Delta_{\mathrm{grid}}/2$, where $\Delta_{\mathrm{grid}}$ is the minimum width among gaps actually occupied by grid values; in general, $\Gamma$ should be computed directly.}
\label{tab:stability}
\centering
\begin{tabular}{ccc}
\hline
$m$ & Minimal gap $\Delta$ & $\Delta/2$ (for reference)\\
\hline
4 & 0.293 & 0.146\\
5 & 0.363 & 0.182\\
6 & 0.134 & 0.067\\
7 & 0.193 & 0.097\\
8 & 0.076 & 0.038\\
\hline
\end{tabular}
\end{table}

The stability threshold exhibits a distinctive alternating pattern because even-sided polygons contain diameter chords that create smaller gaps between critical scales; Figure~\ref{fig:breathing-polygon-geometry} illustrates this geometric mechanism.

% Breathing Polygon Geometry Figure for Section 3.2
% Requires: \usepackage{tikz}, \usepackage{subcaption}, \usetikzlibrary{calc}

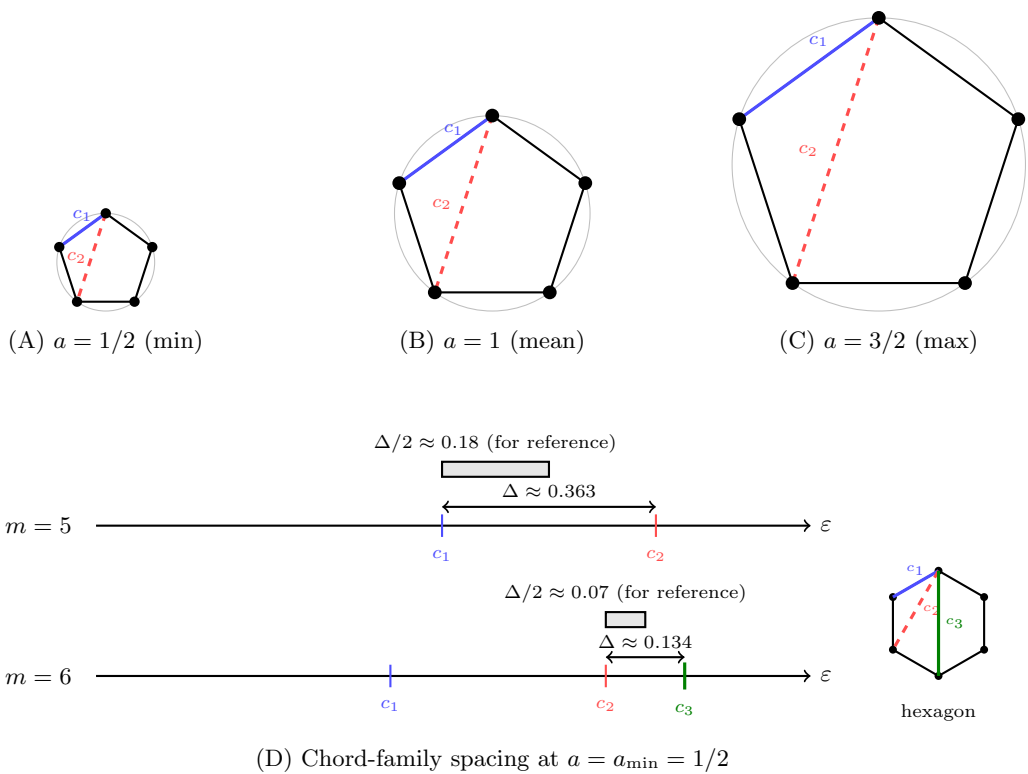
\begin{figure}[ht!]
\centering

% === Common parameters ===
\def\PointSize{2pt}

% Chord length constants (precise values)
\def\pentConeUnit{1.1755705045849463}  % 2*sin(pi/5)
\def\pentCtwoUnit{1.9021130325903071}  % 2*sin(2*pi/5)
\def\hexConeUnit{1.0}                   % 2*sin(pi/6)
\def\hexCtwoUnit{1.7320508075688772}    % 2*sin(2*pi/6) = sqrt(3)
\def\hexCthreeUnit{2.0}                 % 2*sin(3*pi/6) = 2 (diameter)

% === Panel A: Pentagon at minimum radius (a = 0.5) ===
\begin{subfigure}[t]{0.32\textwidth}
\centering
\begin{tikzpicture}[scale=1.3]
  \def\radius{0.5}
  % Circumcircle - now visible
  \draw[gray!50, thin] (0,0) circle[radius=\radius];
  % Pentagon vertices and edges
  \foreach \v in {0,...,4} {
    \coordinate (v\v) at ({90 + \v*72}:\radius);
  }
  \draw[thick] (v0) -- (v1) -- (v2) -- (v3) -- (v4) -- cycle;
  % Highlight c1 (side) and c2 (diagonal)
  \draw[blue!70, very thick] (v0) -- (v1);
  \draw[red!70, very thick, dashed] (v0) -- (v2);
  % Vertices
  \foreach \v in {0,...,4} {
    \fill (v\v) circle[radius=1.5pt];
  }
  % Labels - repositioned for small polygon
  \node[blue!70, font=\scriptsize, above] at ($(v0)!0.5!(v1)$) {$c_1$};
  \node[red!70, font=\scriptsize] at ($(v0)!0.5!(v2) + (-0.15,0)$) {$c_2$};
\end{tikzpicture}
\caption*{\small (A) $a = 1/2$ (min)}
\end{subfigure}
\hfill
%
% === Panel B: Pentagon at mean radius (a = 1) ===
\begin{subfigure}[t]{0.32\textwidth}
\centering
\begin{tikzpicture}[scale=1.3]
  \def\radius{1.0}
  % Circumcircle
  \draw[gray!50, thin] (0,0) circle[radius=\radius];
  % Pentagon vertices and edges
  \foreach \v in {0,...,4} {
    \coordinate (v\v) at ({90 + \v*72}:\radius);
  }
  \draw[thick] (v0) -- (v1) -- (v2) -- (v3) -- (v4) -- cycle;
  % Highlight c1 (side) and c2 (diagonal)
  \draw[blue!70, very thick] (v0) -- (v1);
  \draw[red!70, very thick, dashed] (v0) -- (v2);
  % Vertices
  \foreach \v in {0,...,4} {
    \fill (v\v) circle[radius=\PointSize];
  }
  % Labels
  \node[blue!70, font=\scriptsize] at ($(v0)!0.5!(v1) + (0.08,0.2)$) {$c_1$};
  \node[red!70, font=\scriptsize] at ($(v0)!0.5!(v2) + (-0.22,0)$) {$c_2$};
\end{tikzpicture}
\caption*{\small (B) $a = 1$ (mean)}
\end{subfigure}
\hfill
%
% === Panel C: Pentagon at maximum radius (a = 1.5) ===
\begin{subfigure}[t]{0.32\textwidth}
\centering
\begin{tikzpicture}[scale=1.3]
  \def\radius{1.5}
  % Circumcircle
  \draw[gray!50, thin] (0,0) circle[radius=\radius];
  % Pentagon vertices and edges
  \foreach \v in {0,...,4} {
    \coordinate (v\v) at ({90 + \v*72}:\radius);
  }
  \draw[thick] (v0) -- (v1) -- (v2) -- (v3) -- (v4) -- cycle;
  % Highlight c1 (side) and c2 (diagonal)
  \draw[blue!70, very thick] (v0) -- (v1);
  \draw[red!70, very thick, dashed] (v0) -- (v2);
  % Vertices
  \foreach \v in {0,...,4} {
    \fill (v\v) circle[radius=\PointSize];
  }
  % Labels
  \node[blue!70, font=\scriptsize] at ($(v0)!0.5!(v1) + (0.1,0.28)$) {$c_1$};
  \node[red!70, font=\scriptsize] at ($(v0)!0.5!(v2) + (-0.28,0)$) {$c_2$};
\end{tikzpicture}
\caption*{\small (C) $a = 3/2$ (max)}
\end{subfigure}

\vspace{1.5em}

% === Panel D: Chord-spacing comparison at a_min = 0.5 ===
\begin{subfigure}[t]{\textwidth}
\centering
\begin{tikzpicture}[scale=1.0]
  \def\amin{0.5}
  \def\axislen{9}
  \def\epsmax{1.15}  % scale factor for axis
  
  % Compute chord lengths at a_min for pentagon (m=5)
  \pgfmathsetmacro{\pentCone}{\pentConeUnit * \amin}  % ≈ 0.588
  \pgfmathsetmacro{\pentCtwo}{\pentCtwoUnit * \amin}  % ≈ 0.951
  \pgfmathsetmacro{\pentDelta}{\pentCtwo - \pentCone} % ≈ 0.363
  
  % Compute chord lengths at a_min for hexagon (m=6)
  \pgfmathsetmacro{\hexCone}{\hexConeUnit * \amin}    % = 0.5
  \pgfmathsetmacro{\hexCtwo}{\hexCtwoUnit * \amin}    % ≈ 0.866
  \pgfmathsetmacro{\hexCthree}{\hexCthreeUnit * \amin} % = 1.0
  \pgfmathsetmacro{\hexDeltaSmall}{\hexCthree - \hexCtwo} % ≈ 0.134
  
  % Scaling macro (fixed version)
  \def\toX#1{#1/\epsmax*\axislen}
  
  % --- Top row: Pentagon (m=5) ---
  \def\yTop{1.4}
  
  % Axis
  \draw[thick, ->] (0, \yTop) -- (\axislen + 0.5, \yTop) node[right, font=\small] {$\varepsilon$};
  \node[left, font=\small] at (-0.2, \yTop) {$m=5$};
  
  % Tick at c1
  \draw[blue!70, thick] ({\toX{\pentCone}}, \yTop - 0.15) -- ({\toX{\pentCone}}, \yTop + 0.15);
  \node[below, font=\scriptsize, blue!70] at ({\toX{\pentCone}}, \yTop - 0.2) {$c_1$};
  
  % Tick at c2
  \draw[red!70, thick] ({\toX{\pentCtwo}}, \yTop - 0.15) -- ({\toX{\pentCtwo}}, \yTop + 0.15);
  \node[below, font=\scriptsize, red!70] at ({\toX{\pentCtwo}}, \yTop - 0.2) {$c_2$};
  
  % Delta bracket (main gap)
  \draw[thick, <->] ({\toX{\pentCone}}, \yTop + 0.25) -- ({\toX{\pentCtwo}}, \yTop + 0.25);
  \node[above, font=\scriptsize] at ({(\toX{\pentCone} + \toX{\pentCtwo})/2}, \yTop + 0.25) 
    {$\Delta \approx 0.363$};
  
  % Delta/2 indicator - shown as shaded region within Delta
  \pgfmathsetmacro{\pentMidpoint}{(\pentCone + \pentCtwo)/2}
  \fill[black!10] ({\toX{\pentCone}}, \yTop + 0.65) rectangle ({\toX{\pentMidpoint}}, \yTop + 0.85);
  \draw[thick] ({\toX{\pentCone}}, \yTop + 0.65) rectangle ({\toX{\pentMidpoint}}, \yTop + 0.85);
  \node[above, font=\scriptsize] at ({(\toX{\pentCone} + \toX{\pentMidpoint})/2}, \yTop + 0.85) 
    {$\Delta/2 \approx 0.18$ (for reference)};
  
  % --- Bottom row: Hexagon (m=6) ---
  \def\yBot{-0.6}
  
  % Axis
  \draw[thick, ->] (0, \yBot) -- (\axislen + 0.5, \yBot) node[right, font=\small] {$\varepsilon$};
  \node[left, font=\small] at (-0.2, \yBot) {$m=6$};
  
  % Tick at c1
  \draw[blue!70, thick] ({\toX{\hexCone}}, \yBot - 0.15) -- ({\toX{\hexCone}}, \yBot + 0.15);
  \node[below, font=\scriptsize, blue!70] at ({\toX{\hexCone}}, \yBot - 0.2) {$c_1$};
  
  % Tick at c2
  \draw[red!70, thick] ({\toX{\hexCtwo}}, \yBot - 0.15) -- ({\toX{\hexCtwo}}, \yBot + 0.15);
  \node[below, font=\scriptsize, red!70] at ({\toX{\hexCtwo}}, \yBot - 0.2) {$c_2$};
  
  % Tick at c3 (diameter) - emphasized
  \draw[green!50!black, very thick] ({\toX{\hexCthree}}, \yBot - 0.18) -- ({\toX{\hexCthree}}, \yBot + 0.18);
  \node[below, font=\scriptsize, green!50!black] at ({\toX{\hexCthree}}, \yBot - 0.25) {$c_3$};

  % === Hexagon inset (right side of Panel D) ===
  \begin{scope}[shift={(\axislen + 2.2, {(\yTop + \yBot)/2 - 0.3})}]
    \def\hexRad{0.7}
    % Hexagon vertices
    \foreach \v in {0,...,5} {
      \coordinate (h\v) at ({90 + \v*60}:\hexRad);
    }
    % Hexagon edges (black)
    \draw[thick] (h0) -- (h1) -- (h2) -- (h3) -- (h4) -- (h5) -- cycle;
    % Vertices
    \foreach \v in {0,...,5} {
      \fill (h\v) circle[radius=1.5pt];
    }
    % c1: side (blue) - one edge
    \draw[blue!70, very thick] (h0) -- (h1);
    \node[blue!70, font=\tiny] at ($(h0)!0.5!(h1) + (0.0,0.18)$) {$c_1$};
    % c2: skip-one diagonal (red dashed)
    \draw[red!70, very thick, dashed] (h0) -- (h2);
    \node[red!70, font=\tiny] at ($(h0)!0.5!(h2) + (0.22,0)$) {$c_2$};
    % c3: diameter (green) - across center
    \draw[green!50!black, very thick] (h0) -- (h3);
    \node[green!50!black, font=\tiny] at ($(h0)!0.5!(h3) + (0.22,0)$) {$c_3$};
    % Label
    \node[font=\scriptsize, below] at (0, -\hexRad - 0.25) {hexagon};
  \end{scope}
  
  % Small gap bracket (c2 to c3) - this is Delta for m=6
  \draw[thick, <->] ({\toX{\hexCtwo}}, \yBot + 0.25) -- ({\toX{\hexCthree}}, \yBot + 0.25);
  \node[above, font=\scriptsize] at ({(\toX{\hexCtwo} + \toX{\hexCthree})/2}, \yBot + 0.25) 
    {$\Delta \approx 0.134$};
  
  % Delta/2 indicator - shown as shaded region
  \pgfmathsetmacro{\hexMidpoint}{(\hexCtwo + \hexCthree)/2}
  \fill[black!10] ({\toX{\hexCtwo}}, \yBot + 0.65) rectangle ({\toX{\hexMidpoint}}, \yBot + 0.85);
  \draw[thick] ({\toX{\hexCtwo}}, \yBot + 0.65) rectangle ({\toX{\hexMidpoint}}, \yBot + 0.85);
  \node[above, font=\scriptsize] at ({(\toX{\hexCtwo} + \toX{\hexMidpoint})/2}, \yBot + 0.85) 
    {$\Delta/2 \approx 0.07$ (for reference)};

\end{tikzpicture}
\caption*{\small (D) Chord-family spacing at $a = a_{\min} = 1/2$}
\end{subfigure}

\caption{Geometry of the breathing pentagon and the origin of stability bounds. \textbf{(A)--(C)}~A regular pentagon at minimum, mean, and maximum circumradius $a(t) = 1 + \frac{1}{2}\sin t$; the side length $c_1$ (blue) and diagonal $c_2$ (red dashed) scale proportionally with $a(t)$. \textbf{(D)}~At the minimum radius $a = 1/2$ (where the gap $\Delta$ is smallest), we compare the chord-family spacing for $m=5$ versus $m=6$. The pentagon has two chord families with gap $\Delta \approx 0.363$; the shaded region shows $\Delta/2 \approx 0.18$ for scale. The hexagon (inset) has three chord families: the diameter chord $c_3$ (green) compresses the smallest gap to $\Delta \approx 0.134$, with $\Delta/2 \approx 0.07$. This explains the alternating pattern in Table~\ref{tab:stability}: even-sided polygons have smaller values of $\Delta$ because diameter chords create tighter chord-family spacing. Note that the actual grid clearance $\Gamma$ depends on where grid values land relative to critical distances and should be computed directly.}
\label{fig:breathing-polygon-geometry}
\end{figure}

Figure~\ref{fig:breathing-polygon-geometry} illustrates the geometric origin of the stability bounds. For a regular $m$-gon, there are only $\lfloor m/2 \rfloor$ distinct chord-length families $c_\ell(t) = 2a(t)\sin(\pi\ell/m)$, and the Vietoris--Rips complex can change topology only when the scale parameter $\varepsilon$ crosses one of these critical values. Because every $c_\ell(t)$ scales linearly with the breathing radius $a(t)$, the minimum separation $\Delta$ between consecutive families is attained when $a(t)$ is smallest---that is, at $t = 3\pi/2$ where $a = 1/2$. Panel~(D) reveals why even-sided polygons yield smaller stability thresholds: the diameter chord $c_{m/2}$ introduces an additional family that compresses the critical-value gaps.

To apply Theorem~\ref{thm:exact_stability}, one must compute the grid clearance $\Gamma$ for a specific choice of scale grid $\{\widehat{\varepsilon}_j\}$. For illustration, consider the breathing pentagon ($m=5$) with grid values $\widehat{\varepsilon}_j = 0.2j$ for $j=1,\ldots,15$, uniformly spanning $[0,3]$. At each sampled time $t_i$, the critical distances are $c_1(t_i) = 2a(t_i)\sin(\pi/5)$ and $c_2(t_i) = 2a(t_i)\sin(2\pi/5)$. The cell clearance $g_{i,j}$ measures how far each grid value lies from these moving critical curves. Computing $\Gamma = \min_{i,j} g_{i,j}$ over the time discretization yields the actual stability threshold $\delta_{\max} = \Gamma/2$ for this specific grid. Since the critical distances scale with $a(t)$ while the grid remains fixed, $\Gamma$ is typically smaller than $\Delta/2$ in this example, though as noted in Remark~\ref{rem:gridclearance}, no universal inequality relates these quantities.

Figure~\ref{fig:crocker-stability} provides an illustrative Crocker-diagram view of Theorem~\ref{thm:exact_stability} for the breathing pentagon ($m=5$). The Crocker diagram records $\beta_1(\varepsilon,t)$ on a discrete scale--time grid. In the unperturbed case, $\beta_1=1$ precisely when the scale lies between the two critical distances $c_1(t)$ and $c_2(t)$, producing the band structure in panel~(A).

The exact-stability guarantee is that if $\delta < \Gamma/2$ (where $\Gamma$ is computed for the chosen grid), the Crocker diagram is unchanged cell-by-cell; panel~(B) therefore matches panel~(A). Panel~(C) illustrates how changes become possible once the exact-stability condition is violated: when perturbations cause critical distances to cross grid values, Betti counts change at specific grid cells, which panel~(D) localizes. These visualizations demonstrate the theory in a controlled setting where the distance spectrum has well-separated families.

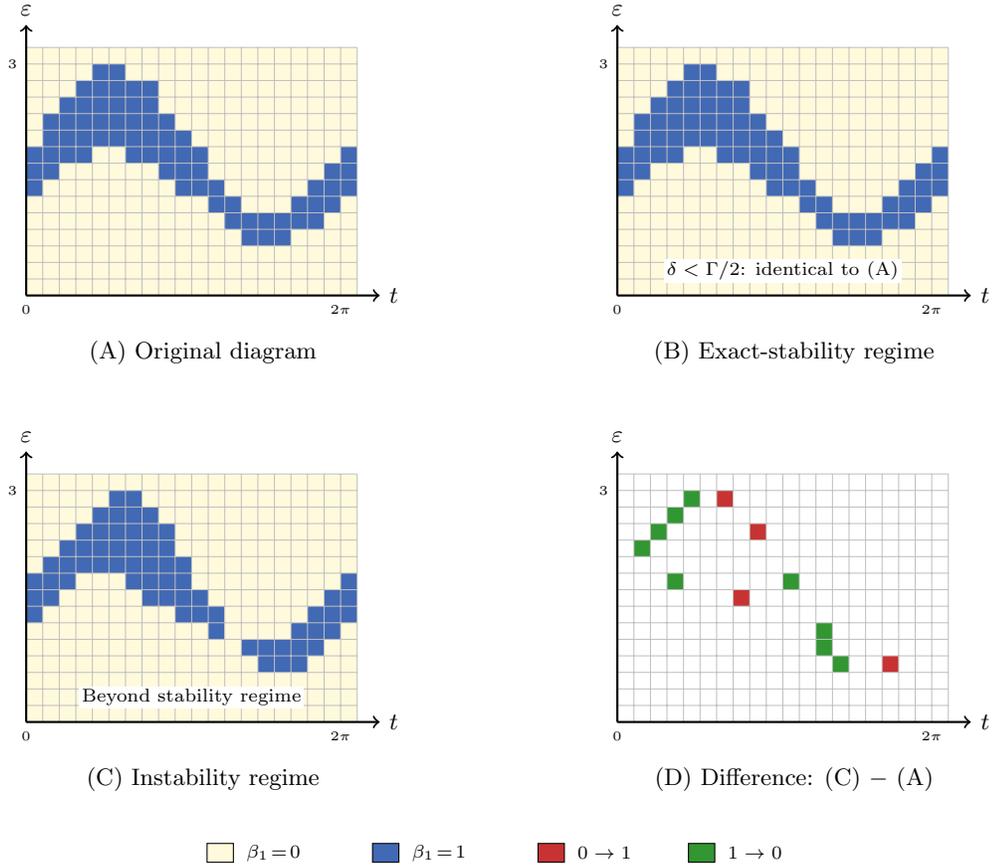
\begin{figure}[ht!]
\centering

% Color definitions
\definecolor{betti0}{RGB}{255,250,220}    % light cream for beta_1 = 0
\definecolor{betti1}{RGB}{65,105,180}     % royal blue for beta_1 = 1
\definecolor{diffhi}{RGB}{200,50,50}      % red for newly = 1 (was 0)
\definecolor{difflo}{RGB}{50,150,50}      % green for newly = 0 (was 1)

% Grid parameters
\def\nx{20}      % number of time steps
\def\ny{15}      % number of epsilon levels
\def\cellw{0.22} % cell width
\def\cellh{0.22} % cell height
\def\epsmax{3.0} % maximum epsilon value
\def\twopi{6.283185307179586} % 2*pi

% Precise chord constants
\def\cOneConst{1.1755705045849463} % 2*sin(pi/5)
\def\cTwoConst{1.9021130325903071} % 2*sin(2*pi/5)

% Perturbation amplitude: 2*0.16 = 0.32 < min gap ≈ 0.3633
\def\perturbAmp{0.16}

% === Panel A: Original Crocker Diagram ===
\begin{subfigure}[t]{0.48\textwidth}
\centering
\begin{tikzpicture}[scale=1]
  \foreach \i in {0,...,19} {
    \foreach \j in {0,...,14} {
      % Time: map i to [0, 2*pi]
      \pgfmathsetmacro{\tval}{\twopi*\i/(\nx-1)}
      \pgfmathsetmacro{\aval}{1 + 0.5*sin(\tval r)}
      \pgfmathsetmacro{\cone}{\cOneConst*\aval}
      \pgfmathsetmacro{\ctwo}{\cTwoConst*\aval}
      % Epsilon: map j to [0, epsmax]
      \pgfmathsetmacro{\epsval}{\epsmax*\j/(\ny-1)}
      \pgfmathtruncatemacro{\betaone}{and(\epsval > \cone, \epsval < \ctwo)}
      \ifnum\betaone=1
        \fill[betti1] (\i*\cellw, \j*\cellh) rectangle +(\cellw, \cellh);
      \else
        \fill[betti0] (\i*\cellw, \j*\cellh) rectangle +(\cellw, \cellh);
      \fi
    }
  }
  
  % Draw grid lines
  \draw[gray!50, very thin] (0,0) grid[step=\cellw] (20*\cellw, 15*\cellh);
  
  % Axes
  \draw[thick, ->] (0,0) -- (20*\cellw + 0.3, 0) node[right] {\small $t$};
  \draw[thick, ->] (0,0) -- (0, 15*\cellh + 0.3) node[above] {\small $\varepsilon$};
  
  % Tick labels (2π at column 19, not at grid boundary)
  \node[below, font=\tiny] at (0,0) {$0$};
  \node[below, font=\tiny] at (19*\cellw,0) {$2\pi$};
  \node[left, font=\tiny] at (0, 14*\cellh) {$3$};
  
\end{tikzpicture}
\caption*{\small (A) Original diagram}
\end{subfigure}
\hfill
%
% === Panel B: Stable Perturbation (IDENTICAL to A) ===
\begin{subfigure}[t]{0.48\textwidth}
\centering
\begin{tikzpicture}[scale=1]
  \foreach \i in {0,...,19} {
    \foreach \j in {0,...,14} {
      \pgfmathsetmacro{\tval}{\twopi*\i/(\nx-1)}
      \pgfmathsetmacro{\aval}{1 + 0.5*sin(\tval r)}
      \pgfmathsetmacro{\cone}{\cOneConst*\aval}
      \pgfmathsetmacro{\ctwo}{\cTwoConst*\aval}
      \pgfmathsetmacro{\epsval}{\epsmax*\j/(\ny-1)}
      \pgfmathtruncatemacro{\betaone}{and(\epsval > \cone, \epsval < \ctwo)}
      \ifnum\betaone=1
        \fill[betti1] (\i*\cellw, \j*\cellh) rectangle +(\cellw, \cellh);
      \else
        \fill[betti0] (\i*\cellw, \j*\cellh) rectangle +(\cellw, \cellh);
      \fi
    }
  }
  
  % Draw grid lines
  \draw[gray!50, very thin] (0,0) grid[step=\cellw] (20*\cellw, 15*\cellh);
  
  % Axes
  \draw[thick, ->] (0,0) -- (20*\cellw + 0.3, 0) node[right] {\small $t$};
  \draw[thick, ->] (0,0) -- (0, 15*\cellh + 0.3) node[above] {\small $\varepsilon$};
  
  % Tick labels
  \node[below, font=\tiny] at (0,0) {$0$};
  \node[below, font=\tiny] at (19*\cellw,0) {$2\pi$};
  \node[left, font=\tiny] at (0, 14*\cellh) {$3$};
  
  % Label indicating stability
  \node[font=\scriptsize, fill=white, inner sep=1pt] at (10*\cellw, 1.5*\cellh) 
    {$\delta < \Gamma/2$: identical to (A)};
    
\end{tikzpicture}
\caption*{\small (B) Exact-stability regime}
\end{subfigure}

\vspace{1em}

% === Panel C: Perturbed (changes visible) ===
\begin{subfigure}[t]{0.48\textwidth}
\centering
\begin{tikzpicture}[scale=1]
  \foreach \i in {0,...,19} {
    \foreach \j in {0,...,14} {
      \pgfmathsetmacro{\tval}{\twopi*\i/(\nx-1)}
      \pgfmathsetmacro{\aval}{1 + 0.5*sin(\tval r)}
      \pgfmathsetmacro{\perturb}{\perturbAmp*sin(2*\tval r)}
      \pgfmathsetmacro{\cone}{\cOneConst*\aval + \perturb}
      \pgfmathsetmacro{\ctwo}{\cTwoConst*\aval - \perturb}
      \pgfmathsetmacro{\epsval}{\epsmax*\j/(\ny-1)}
      \pgfmathtruncatemacro{\betaone}{and(\epsval > \cone, \epsval < \ctwo)}
      \ifnum\betaone=1
        \fill[betti1] (\i*\cellw, \j*\cellh) rectangle +(\cellw, \cellh);
      \else
        \fill[betti0] (\i*\cellw, \j*\cellh) rectangle +(\cellw, \cellh);
      \fi
    }
  }
  
  % Draw grid lines
  \draw[gray!50, very thin] (0,0) grid[step=\cellw] (20*\cellw, 15*\cellh);
  
  % Axes
  \draw[thick, ->] (0,0) -- (20*\cellw + 0.3, 0) node[right] {\small $t$};
  \draw[thick, ->] (0,0) -- (0, 15*\cellh + 0.3) node[above] {\small $\varepsilon$};
  
  % Tick labels
  \node[below, font=\tiny] at (0,0) {$0$};
  \node[below, font=\tiny] at (19*\cellw,0) {$2\pi$};
  \node[left, font=\tiny] at (0, 14*\cellh) {$3$};
  
  % Label indicating instability
  \node[font=\scriptsize, fill=white, inner sep=1pt] at (10*\cellw, 1.5*\cellh) 
    {Beyond stability regime};
    
\end{tikzpicture}
\caption*{\small (C) Instability regime}
\end{subfigure}
\hfill
%
% === Panel D: Difference heatmap ===
\begin{subfigure}[t]{0.48\textwidth}
\centering
\begin{tikzpicture}[scale=1]
  \foreach \i in {0,...,19} {
    \foreach \j in {0,...,14} {
      \pgfmathsetmacro{\tval}{\twopi*\i/(\nx-1)}
      \pgfmathsetmacro{\aval}{1 + 0.5*sin(\tval r)}
      
      % Original values
      \pgfmathsetmacro{\coneOrig}{\cOneConst*\aval}
      \pgfmathsetmacro{\ctwoOrig}{\cTwoConst*\aval}
      
      % Perturbed values
      \pgfmathsetmacro{\perturb}{\perturbAmp*sin(2*\tval r)}
      \pgfmathsetmacro{\conePert}{\cOneConst*\aval + \perturb}
      \pgfmathsetmacro{\ctwoPert}{\cTwoConst*\aval - \perturb}
      
      \pgfmathsetmacro{\epsval}{\epsmax*\j/(\ny-1)}
      
      % Original beta_1
      \pgfmathtruncatemacro{\betaOrig}{and(\epsval > \coneOrig, \epsval < \ctwoOrig)}
      % Perturbed beta_1
      \pgfmathtruncatemacro{\betaPert}{and(\epsval > \conePert, \epsval < \ctwoPert)}
      
      % Compute difference
      \pgfmathtruncatemacro{\diff}{\betaPert - \betaOrig}
      \ifnum\diff=1
        \fill[diffhi] (\i*\cellw, \j*\cellh) rectangle +(\cellw, \cellh);
      \else
        \ifnum\diff=-1
          \fill[difflo] (\i*\cellw, \j*\cellh) rectangle +(\cellw, \cellh);
        \else
          \fill[white] (\i*\cellw, \j*\cellh) rectangle +(\cellw, \cellh);
        \fi
      \fi
    }
  }
  
  % Draw grid lines
  \draw[gray!50, very thin] (0,0) grid[step=\cellw] (20*\cellw, 15*\cellh);
  
  % Axes
  \draw[thick, ->] (0,0) -- (20*\cellw + 0.3, 0) node[right] {\small $t$};
  \draw[thick, ->] (0,0) -- (0, 15*\cellh + 0.3) node[above] {\small $\varepsilon$};
  
  % Tick labels
  \node[below, font=\tiny] at (0,0) {$0$};
  \node[below, font=\tiny] at (19*\cellw,0) {$2\pi$};
  \node[left, font=\tiny] at (0, 14*\cellh) {$3$};
  
\end{tikzpicture}
\caption*{\small (D) Difference: (C) $-$ (A)}
\end{subfigure}

\vspace{0.5em}

% === Legend ===
\centering
\begin{tikzpicture}
  \fill[betti0] (0,0) rectangle (0.35,0.25);
  \draw (0,0) rectangle (0.35,0.25);
  \node[right, font=\scriptsize] at (0.4,0.125) {$\beta_1\!=\!0$};
  
  \fill[betti1] (2.2,0) rectangle (2.55,0.25);
  \draw (2.2,0) rectangle (2.55,0.25);
  \node[right, font=\scriptsize] at (2.6,0.125) {$\beta_1\!=\!1$};
  
  \fill[diffhi] (4.4,0) rectangle (4.75,0.25);
  \draw (4.4,0) rectangle (4.75,0.25);
  \node[right, font=\scriptsize] at (4.8,0.125) {$0 \to 1$};
  
  \fill[difflo] (6.4,0) rectangle (6.75,0.25);
  \draw (6.4,0) rectangle (6.75,0.25);
  \node[right, font=\scriptsize] at (6.8,0.125) {$1 \to 0$};
\end{tikzpicture}

\caption{Crocker diagrams for the breathing pentagon ($m=5$), displaying $\beta_1(\varepsilon, t)$ on a discrete grid of scales and times. (A)~The unperturbed Crocker diagram shows a characteristic band where $\beta_1=1$, corresponding to the loop that exists when the scale satisfies $c_1(t) < \varepsilon < c_2(t)$, where $c_1(t)=2a(t)\sin(\pi/5)$ and $c_2(t)=2a(t)\sin(2\pi/5)$. (B)~In the exact-stability regime of Theorem~\ref{thm:exact_stability} (i.e., $\delta < \Gamma/2$ where $\Gamma$ is the grid clearance), the Crocker diagram is unchanged; we display an identical diagram to emphasize the cell-for-cell invariance guaranteed by the theorem. (C)~To illustrate how changes become possible once we move beyond the exact-stability regime, we display a perturbed instance in which the critical curves are shifted enough to cross some grid boundaries, producing localized changes in $\beta_1$. (D)~The entrywise difference (C)$-$(A) highlights where grid cells change: red cells gained a loop ($0\to1$) and green cells lost one ($1\to0$). These diagrams are illustrative demonstrations of the stability theory in a controlled setting, not empirical validation from experimental data.}
\label{fig:crocker-stability}
\end{figure}

We now turn to an example motivated by biological experiments, where the structure of the distance spectrum differs fundamentally from the polygon case.

\subsection{Feasibility Analysis for Epithelial Cell Imaging}
\label{sec:cellmigration}

Topological data analysis of collective cell behavior is an ideal proving ground for our stability theory because interpreting experimental results hinges on precise cell position measurements. Epithelial systems---the protective cellular layers that cover surfaces and line cavities in the body---form continuous sheets whose collective migration is critical in wound healing, embryonic development, and pathologies such as cancer metastasis~\cite{Rorth2009}. In this section, we demonstrate how our stability bounds translate into concrete experimental design constraints, with an important caveat: unlike the breathing polygon, large biological point clouds have dense distance spectra that make the exact-stability criterion difficult to verify. We therefore focus on the bounded-change theorem (Theorem~\ref{thm:global_bound}) as the appropriate tool for this setting.

Figure~\ref{fig:epithelial-schematic} provides geometric intuition for the analysis below: $\delta$ represents centroid localization uncertainty, while the local density parameter $\Lambda_\delta$ controls how topological changes propagate.

% Epithelial Sheet Schematic Figure for Section 3.3 (horizontal, minimal on-figure text)
% Place after the introductory paragraph, before parameter estimation
% Requires: \usepackage{tikz}, \usepackage{subcaption}, \usetikzlibrary{calc,decorations.pathmorphing}

\begin{figure}[ht!]
\centering

% ---------------------------
% Panel A: Epithelial sheet
% ---------------------------
\begin{subfigure}[t]{0.31\textwidth}
\centering
\begin{tikzpicture}[scale=0.85]
  % Field of view (match across panels)
  \draw[thick, gray!70, rounded corners=2pt] (0,0) rectangle (3.0,3.6);

  % Internal "cell boundary" network (print-safe)
  \draw[gray!65, decorate, decoration={random steps, segment length=6pt, amplitude=0.7pt}]
      (1.0,0.0) -- (1.0,3.6);
  \draw[gray!65, decorate, decoration={random steps, segment length=6pt, amplitude=0.7pt}]
      (2.0,0.0) -- (2.0,3.6);

  \draw[gray!65, decorate, decoration={random steps, segment length=6pt, amplitude=0.7pt}]
      (0.0,1.0) -- (3.0,1.0);
  \draw[gray!65, decorate, decoration={random steps, segment length=6pt, amplitude=0.7pt}]
      (0.0,2.0) -- (3.0,2.0);
  \draw[gray!65, decorate, decoration={random steps, segment length=6pt, amplitude=0.7pt}]
      (0.0,3.0) -- (3.0,3.0);

  % Small local sparse region ("gap") — schematic
  \fill[white] (1.18,1.58) rectangle (1.82,2.22);
  \draw[gray!65] (1.18,1.58) rectangle (1.82,2.22);
  \node[gray!65, font=\scriptsize] at (1.50,1.90) {gap};

  % Centroids (faint) — SAME coordinates used in panels B and C
  \coordinate (L0) at (0.35,0.60);
  \coordinate (M0) at (1.50,0.50);
  \coordinate (R0) at (2.65,0.60);

  \coordinate (P1) at (0.95,1.35);
  \coordinate (P2) at (2.05,1.35);
  \coordinate (P3) at (2.05,2.45);
  \coordinate (P4) at (0.95,2.45);

  \coordinate (L1) at (0.35,1.90);
  \coordinate (R1) at (2.65,1.90);

  \coordinate (L2) at (0.35,3.20);
  \coordinate (M2) at (1.50,3.30);
  \coordinate (R2) at (2.65,3.20);

  \foreach \pt in {L0,M0,R0,P1,P2,P3,P4,L1,R1,L2,M2,R2} {
    \fill[black!45] (\pt) circle[radius=1.7pt];
  }

\end{tikzpicture}
\caption*{\small (A) Epithelial sheet (schematic)}
\end{subfigure}\hfill
%
% ---------------------------
% Panel B: Centroids + delta (no on-figure text)
% ---------------------------
\begin{subfigure}[t]{0.31\textwidth}
\centering
\begin{tikzpicture}[scale=0.85]
  % Same field of view (match Panel A)
  \draw[gray!55, dashed] (0,0) rectangle (3.0,3.6);

  % Same centroid set as Panel A
  \coordinate (L0) at (0.35,0.60);
  \coordinate (M0) at (1.50,0.50);
  \coordinate (R0) at (2.65,0.60);

  \coordinate (P1) at (0.95,1.35);
  \coordinate (P2) at (2.05,1.35);
  \coordinate (P3) at (2.05,2.45);
  \coordinate (P4) at (0.95,2.45);

  \coordinate (L1) at (0.35,1.90);
  \coordinate (R1) at (2.65,1.90);

  \coordinate (L2) at (0.35,3.20);
  \coordinate (M2) at (1.50,3.30);
  \coordinate (R2) at (2.65,3.20);

  \foreach \pt in {L0,M0,R0,P1,P2,P3,P4,L1,R1,L2,M2,R2} {
    \fill[black] (\pt) circle[radius=2.6pt];
  }

  % Example localization uncertainty circles (kept light; meaning explained in caption)
  \draw[gray!45, thin] (P2) circle[radius=0.18];
  \draw[gray!45, thin] (M2) circle[radius=0.18];

\end{tikzpicture}
\caption*{\small (B) Cell centroids (schematic)}
\end{subfigure}\hfill
%
% ---------------------------
% Panel C: VR complex + 1-cycle (no epsilon legend; meaning explained in caption)
% ---------------------------
\begin{subfigure}[t]{0.31\textwidth}
\centering
\begin{tikzpicture}[scale=0.85]
  % Same centroid set
  \coordinate (L0) at (0.35,0.60);
  \coordinate (M0) at (1.50,0.50);
  \coordinate (R0) at (2.65,0.60);

  \coordinate (P1) at (0.95,1.35);
  \coordinate (P2) at (2.05,1.35);
  \coordinate (P3) at (2.05,2.45);
  \coordinate (P4) at (0.95,2.45);

  \coordinate (L1) at (0.35,1.90);
  \coordinate (R1) at (2.65,1.90);

  \coordinate (L2) at (0.35,3.20);
  \coordinate (M2) at (1.50,3.30);
  \coordinate (R2) at (2.65,3.20);

  % Edges supporting triangles (draw first, light gray)
  \draw[gray!70] (M0) -- (P1);
  \draw[gray!70] (M0) -- (P2);
  \draw[gray!70] (L1) -- (P1);
  \draw[gray!70] (L1) -- (P4);
  \draw[gray!70] (R1) -- (P2);
  \draw[gray!70] (R1) -- (P3);
  \draw[gray!70] (M2) -- (P4);
  \draw[gray!70] (M2) -- (P3);

  % Additional peripheral edges (context)
  \draw[gray!70] (L2) -- (P4);
  \draw[gray!70] (R2) -- (P3);

  % Highlighted 1-cycle (4-cycle around the small gap)
  \draw[red!70, very thick] (P1) -- (P2) -- (P3) -- (P4) -- cycle;

  % Fill 2-simplices (triangles), each with all three edges present
  \fill[blue!15] (M0) -- (P1) -- (P2) -- cycle;
  \fill[blue!15] (L1) -- (P1) -- (P4) -- cycle;
  \fill[blue!15] (R1) -- (P2) -- (P3) -- cycle;
  \fill[blue!15] (M2) -- (P3) -- (P4) -- cycle;

  % Points on top
  \foreach \pt in {L0,M0,R0,P1,P2,P3,P4,L1,R1,L2,M2,R2} {
    \fill[black] (\pt) circle[radius=2.6pt];
  }

  % Labels (kept; these are genuinely helpful)
  \node[red!70, font=\scriptsize] at (1.50,1.90) {$1$-cycle};
  \node[red!70, font=\scriptsize] at (1.50,0.15) {$\beta_1=1$};

\end{tikzpicture}
\caption*{\small (C) VR complex at scale $\varepsilon$}
\end{subfigure}

\caption{Schematic pipeline for the feasibility calculation in Section~\ref{sec:cellmigration}. \textbf{(A)}~A cartoon epithelial sheet (confluent patch) with a small local sparse region. \textbf{(B)}~Corresponding centroid locations (schematic); light circles illustrate an example localization uncertainty of radius $\delta$ around selected centroids. \textbf{(C)}~A schematic Vietoris--Rips complex at a chosen scale $\varepsilon$: edges connect centroids whose pairwise distance is at most $\varepsilon$, and filled triangles indicate 2-simplices. A highlighted interior 1-cycle contributes to $\beta_1$.}
\label{fig:epithelial-schematic}
\end{figure}
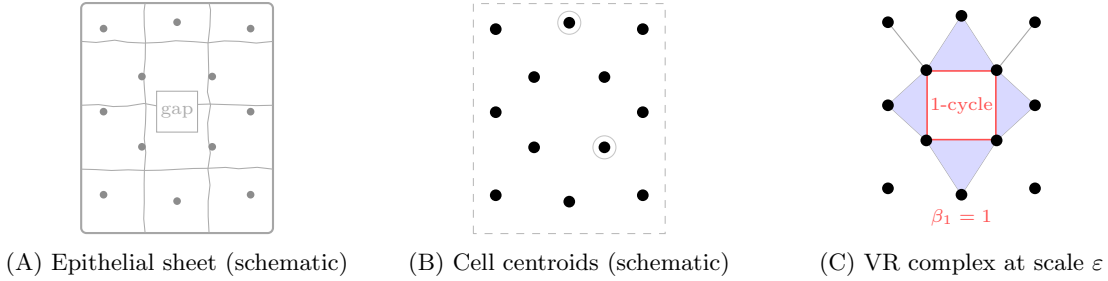

At the tissue scale, epithelial cells do not maintain fixed neighbor relationships. Unlike crystalline lattices, these sheets continually rearrange, so purely geometric summaries---cell density, average spacing, and the like---miss the neighbor-switching events that drive collective behavior. Topological descriptors, by contrast, are built to track changing connectivity. Here we use the first Betti number, $\beta_1$, which counts loops in the centroid network and thus records how voids open or close as cells move. Bonilla et al.~\cite{Bonilla2020} previously demonstrated that $\beta_1$ quantifies the complexity of cell interfaces and the formation of multi-cellular holes as epithelial sheets rearrange.

The exact-stability theorem (Theorem~\ref{thm:exact_stability}) requires computing the grid clearance $\Gamma$, which in turn requires characterizing how the distance spectrum sits relative to the fixed grid. For $m \approx 500$ cells, there are $\binom{500}{2} \approx 125{,}000$ pairwise distances. In a generic disordered configuration, these distances are nearly all distinct, making the minimum gap $\Delta$ between consecutive distances extremely small---potentially comparable to or smaller than the measurement precision. Moreover, the in-gap condition (that each grid value $\widehat{\varepsilon}_j$ lies strictly between consecutive critical distances at every time) becomes difficult to verify and may yield a clearance $\Gamma$ so small that the exact-stability criterion becomes vacuous for practical noise levels. This contrasts sharply with the breathing polygon, where symmetry reduces $\binom{m}{2}$ pairs to only $\lfloor m/2 \rfloor$ distinct distance families, yielding a well-separated spectrum amenable to exact-stability analysis.

For large biological point clouds, the bounded-change theorem (Theorem~\ref{thm:global_bound}) provides the appropriate stability guarantee. This theorem makes no assumptions about the distance spectrum; it bounds the total Betti-number change in terms of local density $\Lambda_\delta$, the number of perturbed points $m^*$, and the grid size.

To apply this bound, we perform a thought experiment based on Park \textit{et al.}'s work~\cite{Park2015}, which imaged human bronchial epithelial sheets every three minutes for 150 minutes, yielding $n_t = 51$ time points. We assume $m \approx 500$ centroids per frame and select Vietoris--Rips scales $\widehat{\varepsilon}_j = 20j\,\mu\text{m}$ for $j=1,\dots,6$, constituting roughly 1--6 cell diameters as in~\cite{Bonilla2020}.

We first estimate the centroid-localization uncertainty $\delta$. Recent tracking benchmarks suggest position errors can be reliably bounded by three pixels under controlled conditions~\cite{Chenouard2014,Ulman2017,Holme2023}. With Park's imaging resolution of $0.44~\mu\text{m}$ per pixel~\cite{Park2015}, this gives:
\begin{equation}
\delta \approx 3 \text{ pixels} \times 0.44~\mu\text{m/pixel} = 1.3~\mu\text{m}.
\end{equation}

Since the bounded-change theorem (Theorem~\ref{thm:global_bound}) uses the $\delta$-inflated density $\Lambda_\delta(\widehat{\varepsilon}_j)$, we must account for the perturbation magnitude. With $\delta \approx 1.3~\mu\text{m}$ and scales $\widehat{\varepsilon}_j = 20j~\mu\text{m}$, the inflation $2\delta \approx 2.6~\mu\text{m}$ is small compared to the scale values (at most 13\% of the smallest scale). For hexagonal packing with cell diameter $\phi \approx 15~\mu\text{m}$, this inflation slightly increases the neighbor counts but does not qualitatively change the analysis. We therefore proceed with the unperturbed estimates $\Lambda(\widehat{\varepsilon}_j) = 1 + 3j(j+1)$ as a reasonable approximation, noting that the true bounds using $\Lambda_\delta$ are slightly larger.

Assuming a hexagonal packing of cells with diameter $\phi \approx 15\,\mu\text{m}$---the densest planar arrangement---the neighbor count within $j$ rings is $\Lambda(\widehat{\varepsilon}_j)=1+3j(j+1)$~\cite{ConwaySloane1999}, giving $7,\,19,\,37,\,61,\,91,$ and $127$ neighbors for $j=1$ through $6$, respectively. Here we treat $j$ as an effective ring index for the purpose of a back-of-the-envelope bound; the exact conversion from the physical scale $\widehat{\varepsilon}_j$ to a ring count depends on the centroid spacing. The hexagonal-packing count is used only as a worst-case local-density envelope, not as a model of the observed tissue geometry.

For $k=1$, the per-cell bound from Proposition~\ref{prop:betti_change} is $\binom{\Lambda_\delta}{2}$. At the largest scale ($j=6$, $\Lambda \approx 127$), this gives $\binom{127}{2} = 8001$ as the worst-case $\beta_1$ change per perturbed vertex per grid cell. Summing over all scales:
\begin{equation}
\sum_{j=1}^{6} \binom{\Lambda(\widehat{\varepsilon}_j)}{2} = \binom{7}{2} + \binom{19}{2} + \binom{37}{2} + \binom{61}{2} + \binom{91}{2} + \binom{127}{2} \approx 14{,}800.
\end{equation}

If every centroid is perturbed ($m^* = 500$) across $n_t = 51$ frames, the global bound becomes approximately $51 \times 500 \times 14{,}800 \approx 3.8 \times 10^8$, or an average of about $1.2 \times 10^6$ per Crocker cell. This worst-case bound is extremely conservative---it assumes every possible simplex change occurs in the worst possible direction.

In practice, topological changes are far more localized. The bound's value lies not in its tightness but in its \emph{scaling}: it depends on local density $\Lambda_\delta$ rather than global point count $m$, explaining why Crocker diagrams remain informative even for large point clouds. Moreover, the bound decreases rapidly at smaller scales where $\Lambda_\delta$ is smaller, suggesting that low-scale topological features are more robust than high-scale ones.

The bounded-change theorem provides concrete guidance for experimental design. Higher imaging resolution reduces $\delta$, which tightens $\Lambda_\delta(\widehat{\varepsilon}_j)$ and thus tightens the worst-case bound; moreover, smaller perturbations also reduce the \emph{actual} number of edge-status changes in practice, improving empirical stability. Crocker grids focused on smaller scales (where $\Lambda_\delta$ is smaller) yield more robust topological summaries, as reflected in the tighter bounds at low $j$. Sparser cell configurations have smaller $\Lambda_\delta$ values and hence tighter stability bounds.

When resolution is limiting, established remedies such as deconvolution-based deblurring~\cite{Sarder2006,Sibarita2005} and super-resolution microscopy~\cite{Betzig2006,Schermelleh2010} can drive the localization error $\delta$ downward, which tightens $\Lambda_\delta$ and improves both theoretical bounds and empirical stability.

In summary, while exact stability (Theorem~\ref{thm:exact_stability}) provides the strongest guarantee, it requires a well-separated distance spectrum that large biological point clouds typically lack. The bounded-change theorem (Theorem~\ref{thm:global_bound}) fills this gap, providing universal guarantees that explain the empirical robustness of Crocker diagrams and guide experimental design even when exact-stability conditions cannot be verified.

\section{Probabilistic Stability to Random Perturbations}
\label{sec:probabilistic}

In this section we develop a probabilistic framework for Crocker diagrams subjected to random Gaussian perturbations of the underlying point cloud. When multiple independent error sources combine---as is common in experimental settings---their cumulative effect tends to follow a Gaussian distribution, as predicted by the central limit theorem. Two key parameters anchor the theory: (1) the noise standard deviation $\sigma$ that quantifies uncertainty, and (2) the grid-threshold clearance $\Gamma_{\mathrm{grid}}$ that measures how far pairwise distances sit from grid values.

We first prove a fundamental stability criterion: if the grid-threshold clearance satisfies $\Gamma_{\mathrm{grid}} > \sqrt{2}\sigma\sqrt{d}$, the probability of topological changes decays exponentially with $(\Gamma_{\mathrm{grid}}/\sigma)^2$. When this condition is met, we derive bounds showing that the total failure probability depends on the number of points, time steps, and scale values in the Crocker grid. As we will show, these results translate statistical measurement errors into confidence levels that can inform experimental design. We revisit our two examples from Section~\ref{sec:deterministic}: first, we add Gaussian noise to the breathing pentagon to demonstrate the probabilistic framework in a setting where $\Gamma_{\mathrm{grid}}$ can be computed directly; then, we discuss why probabilistic guarantees face the same scope limitations as deterministic exact stability for large biological point clouds.

\subsection{Theoretical Stability for Random Perturbations}
\label{sec:probabilistic_theory}

We now extend our stability analysis to the probabilistic setting, examining how random perturbations affect the robustness of Crocker diagrams and quantifying the likelihood of topological changes under measurement noise.

Consider a dynamic point cloud $P(t) = \{p_1(t), \dots, p_m(t)\}$ subject to random perturbations. We model the observed positions as:
\begin{equation}
\tilde{p}_a(t) = p_a(t) + \xi_a(t),
\end{equation}
where each $\xi_a(t)$ is an isotropic Gaussian random vector with zero mean and covariance matrix $\sigma^2 I_d$. This models independent noise of equal magnitude in each coordinate direction, with temporal independence between consecutive observations. Our goal is to determine the probability that these random perturbations cause a change in the Crocker diagram.

A Crocker-diagram entry can change only if, for some sampled time $t_i$ and grid scale $\widehat{\varepsilon}_j$, the perturbation causes at least one critical distance to move across $\widehat{\varepsilon}_j$---equivalently, the ordering of $\widehat{\varepsilon}_j$ relative to the critical-distance set at $t_i$ changes.

For any pair of points $(a,b)$, let $s_{ab}(t) = \|p_a(t) - p_b(t)\|$ denote the original distance and $\tilde{s}_{ab}(t) = \|\tilde{p}_a(t) - \tilde{p}_b(t)\|$ the perturbed distance. By the reverse triangle inequality, the change in distance satisfies:
\begin{equation}
|s_{ab}(t) - \tilde{s}_{ab}(t)| \leq \|\xi_a(t) - \xi_b(t)\|.
\end{equation}
Equality holds when the perturbations are collinear with $p_a(t) - p_b(t)$.

Since the change in distance is bounded by $\|\xi_a(t) - \xi_b(t)\|$, we need to characterize the distribution of this quantity. When $\xi_a(t)$ and $\xi_b(t)$ are independent Gaussian vectors, their difference follows a specific distribution:

\begin{proposition}[Distribution of Perturbation Differences; see, e.g.,~\cite{Vershynin2018}]
\label{prop:perturbation_diff}
Let $\xi_a(t)$ and $\xi_b(t)$ be independent isotropic Gaussian random vectors in $\mathbb{R}^d$ with zero mean and covariance matrix $\sigma^2 I_d$. Then:
\begin{equation}
\frac{\|\xi_a(t) - \xi_b(t)\|}{\sqrt{2}\sigma} \sim \chi_d,
\end{equation}
where $\chi_d$ denotes the chi distribution with $d$ degrees of freedom.
\end{proposition}

From our deterministic analysis, we know that exact cell-by-cell invariance requires that no critical distance cross any grid threshold $\widehat{\varepsilon}_j$, a condition equivalently expressed by $\delta < \Gamma$. In the probabilistic setting, however, this uniform clearance condition may fail ($\Gamma=0$) even though most grid cells are well separated from critical distances. To quantify stability in this regime, we therefore work with cell-level clearance quantities that remain well-defined without requiring a positive global minimum.

\begin{definition}[Grid-Threshold Clearance]
\label{def:prob_clearance}
The \emph{grid-threshold clearance} is
\begin{equation}
\Gamma_{\mathrm{grid}} := \min_{i=1,\ldots,n_t} \min_{j=1,\ldots,n_\varepsilon} \min_{a<b} \bigl| \|p_a(t_i) - p_b(t_i)\| - \widehat{\varepsilon}_j \bigr|,
\end{equation}
the minimum distance from any pairwise distance to any grid threshold, across all sampled times.
\end{definition}

If $\Gamma_{\mathrm{grid}} > 0$ and all pairwise distance changes remain below $\Gamma_{\mathrm{grid}}$, no grid threshold can be crossed, and the Crocker diagram is unchanged. We need to bound $\mathbb{P}[\|\xi_a(t) - \xi_b(t)\| > \Gamma_{\mathrm{grid}}]$ using the $\chi$ distribution. The following concentration inequality gives us the tool we need:

\begin{theorem}[Concentration Bound~\cite{Laurent2000}]
\label{thm:concentration}
For any $\tau \geq 0$, we have
\begin{equation}
\mathbb{P}[\|\xi_a(t) - \xi_b(t)\| \geq \sqrt{2}\sigma(\sqrt{d} + \tau)] \leq e^{-\tau^2/2}.
\end{equation}
\end{theorem}

To apply this theorem to our stability threshold $\Gamma_{\mathrm{grid}}$, we express it in the form $\sqrt{2}\sigma(\sqrt{d} + \tau)$. This leads us to define:
\begin{equation}
\label{eq:taustar}
\tau^\star = \frac{\Gamma_{\mathrm{grid}}}{\sqrt{2}\sigma} - \sqrt{d}.
\end{equation}
This parameter $\tau^\star$ measures how far the grid-threshold clearance $\Gamma_{\mathrm{grid}}$ exceeds the typical perturbation magnitude $\sqrt{2}\sigma\sqrt{d}$, expressed in units of $\sqrt{2}\sigma$. For the probabilistic bound to be meaningful, we need $\tau^\star > 0$, which is equivalent to
\begin{equation}
\label{eq:taustarzero}
\Gamma_{\mathrm{grid}} > \sqrt{2}\sigma\sqrt{d}.
\end{equation}

This inequality has a direct stability interpretation: the grid-threshold clearance must exceed the typical noise-induced fluctuation in distances. To see why $\sqrt{2}\sigma\sqrt{d}$ represents the typical fluctuation, we return to Proposition~\ref{prop:perturbation_diff}, which established that $\|\xi_a(t) - \xi_b(t)\|/(\sqrt{2}\sigma) \sim \chi_d$. The expected value of a chi-distributed random variable is:
\begin{equation}
\mathbb{E}[\|\xi_a(t)-\xi_b(t)\|]
  = \sqrt{2}\sigma\,
    \frac{\mathrm{Gamma}(\frac{d+1}{2})}
         {\mathrm{Gamma}(\frac{d}{2})}
  \approx \sqrt{2}\sigma\sqrt{d}\left(1-\frac{1}{4d}\right),
\end{equation}
where we have written out the word for the Gamma function to avoid confusion with our grid clearance quantities. Our stability condition~\eqref{eq:taustarzero} requires the grid-threshold clearance to exceed this typical fluctuation magnitude.

When $\tau^\star > 0$, the Laurent--Massart bound implies that the probability of any topological change decays exponentially in $(\tau^\star)^2$:

\begin{proposition}[Single-pair crossing probability bound] 
\label{prop:pair_prob} 
If $\Gamma_{\mathrm{grid}} > \sqrt{2}\sigma\sqrt{d}$, then the probability that a single pair of points experiences a distance change exceeding $\Gamma_{\mathrm{grid}}$ is bounded by:
\begin{equation}
p_{\text{cross}} := \mathbb{P}\left[\|\xi_a(t) - \xi_b(t)\| > \Gamma_{\mathrm{grid}}\right] \leq \exp\left[-\frac{(\tau^\star)^2}{2}\right],
\label{eq:singlepair}
\end{equation}
where $\tau^\star$ is defined in~\eqref{eq:taustar}. Since $|\tilde{s}_{ab}(t) - s_{ab}(t)| \le \|\xi_a(t) - \xi_b(t)\|$, any threshold crossing requires $\|\xi_a(t) - \xi_b(t)\| > \Gamma_{\mathrm{grid}}$; we bound the probability of this necessary condition.
\end{proposition}

\begin{proof} 
By definition of $\tau^\star$, we have $\Gamma_{\mathrm{grid}} = \sqrt{2}\sigma(\sqrt{d} + \tau^\star)$. Applying Theorem~\ref{thm:concentration}:
\begin{equation} 
\mathbb{P}\left[\|\xi_a(t) - \xi_b(t)\| > \Gamma_{\mathrm{grid}}\right] = \mathbb{P}\left[\|\xi_a(t) - \xi_b(t)\| > \sqrt{2}\sigma(\sqrt{d} + \tau^\star)\right] \leq e^{-(\tau^\star)^2/2}.
\end{equation} 
\end{proof}

This bound applies to a single pair of points at a single time. For the entire Crocker diagram to remain stable, all pairwise distances must avoid crossing any grid threshold across all time points.

\begin{remark}
The grid-threshold clearance $\Gamma_{\mathrm{grid}}$ is related to the deterministic clearance $\Gamma$ in Definition~\ref{def:clearance}. When the in-gap condition holds and all bracketing values are actual pairwise distances (not the boundary extensions $\varepsilon_0 = 0$ or $\varepsilon_{M+1} = +\infty$), we have $\Gamma_{\mathrm{grid}} = \Gamma$. In general, $\Gamma_{\mathrm{grid}} \ge \Gamma$, since $\Gamma$ may use boundary values that are not actual pairwise distances. However, $\Gamma_{\mathrm{grid}}$ is always well-defined, even when some grid values lie outside the range of critical distances. For precise guarantees, $\Gamma_{\mathrm{grid}}$ should be computed directly from the data and grid specification.
\end{remark}

Let $\mathcal{E}$ denote the event that the Crocker diagram changes due to random perturbations. We now establish a global probabilistic bound:

\begin{theorem}[Global Probabilistic Stability]
\label{thm:global_prob}
Let $P(t)$ be a dynamic point cloud with $m$ points observed at $n_t$ time points, and let its Crocker diagram be computed using $n_\varepsilon$ scale values. Suppose each point is perturbed independently at each time step by isotropic Gaussian noise with variance $\sigma^2$. If $\Gamma_{\mathrm{grid}} > \sqrt{2}\sigma\sqrt{d}$, then the probability that the Crocker diagram changes due to noise is bounded by
\begin{equation}
\mathbb{P}[\mathcal{E}] \leq \min\left\{1, \binom{m}{2} n_t \exp\left[-\frac{(\tau^\star)^2}{2}\right] \right\},
\label{eq:prob-stability-bound}
\end{equation}
with $\tau^\star$ as in \eqref{eq:taustar}.
\end{theorem}

\begin{proof}
A change in the Crocker diagram can occur only if some pairwise distance crosses a grid threshold at some time $t_i$. By Definition~\ref{def:prob_clearance}, every pairwise distance is at least $\Gamma_{\mathrm{grid}}$ away from the nearest grid threshold. Thus, a threshold crossing for pair $(a,b)$ at time $t_i$ requires $\|\xi_a(t_i) - \xi_b(t_i)\| \geq \Gamma_{\mathrm{grid}}$. By Proposition~\ref{prop:pair_prob}, this event has probability at most $\exp[-(\tau^\star)^2/2]$. Applying the union bound over $\binom{m}{2}$ pairs and $n_t$ time points:
\begin{equation}
\mathbb{P}[\mathcal{E}] \leq \binom{m}{2} n_t \exp\left[-\frac{(\tau^\star)^2}{2}\right].
\end{equation}
Since probabilities cannot exceed one, we take the minimum with 1.
\end{proof}

The union bound is conservative because crossings at different thresholds are not independent, and because many potential crossings do not actually change Betti numbers. Our bound controls the event that any threshold crossing occurs; absence of threshold crossings is sufficient (but not necessary) for invariance.

The expression \eqref{eq:prob-stability-bound} reveals key insights about stability when $\tau^\star > 0$. First, the failure probability contains $\exp[-(\tau^\star)^2/2]$, showing exponential decay as the clearance $\Gamma_{\mathrm{grid}}$ increases or noise $\sigma$ decreases. Second, the $\sqrt{d}$ term in $\tau^\star$ shows that higher dimension demands proportionally larger clearance to maintain stability. Finally, the prefactor $\binom{m}{2} n_t$ indicates that more complex datasets---with more points or time samples---provide more opportunities for instability.

Together with our deterministic stability results in Section~\ref{sec:stability_theory}, this probabilistic analysis provides a theoretical foundation for understanding the robustness of Crocker diagrams under measurement noise. The exponential decay in the probability bound explains why Crocker diagrams can remain effective even in the presence of noise, provided the grid-threshold clearance is sufficiently large.

\subsection{Analytical Application: Breathing Pentagon with Noise}

To illustrate our probabilistic framework, we return to the breathing pentagon ($m=5$) in $\mathbb{R}^2$ ($d=2$). In Section~\ref{sec:breathing-poly}, we established that the minimum gap between consecutive critical distances is $\Delta \approx 0.363$, occurring at $t = 3\pi/2$ when $a(t) = 1/2$. For probabilistic stability, we must compute the grid-threshold clearance $\Gamma_{\mathrm{grid}}$ (Definition~\ref{def:prob_clearance}).

For the breathing pentagon, the critical distances are $c_1(t) = 2a(t)\sin(\pi/5)$ and $c_2(t) = 2a(t)\sin(2\pi/5)$. As $a(t)$ varies from $1/2$ to $3/2$, these critical distances sweep through the ranges approximately $[0.59, 1.76]$ and $[0.95, 2.85]$ respectively. We select a scale grid with $n_\varepsilon = 15$ values: $\widehat{\varepsilon}_j = 0.1j$ for $j = 1, \ldots, 15$, spanning $[0.1, 1.5]$. We sample $n_t = 51$ evenly spaced time points over $[0, 2\pi]$.

For this specific grid and time discretization, we compute $\Gamma_{\mathrm{grid}}$ by evaluating, for each sampled time $t_i$ and each grid value $\widehat{\varepsilon}_j$, the minimum distance from any pairwise distance to $\widehat{\varepsilon}_j$. Because the regular pentagon has only two distinct pairwise distances at each time ($c_1(t)$ and $c_2(t)$), computing $\Gamma_{\mathrm{grid}}$ reduces to minimizing the distance from each $\widehat{\varepsilon}_j$ to these two curves over sampled times. Numerically, we find:
\begin{equation}
\Gamma_{\mathrm{grid}} \approx 0.032.
\end{equation}
This value is much smaller than $\Delta/2 \approx 0.18$ because the critical distances $c_1(t)$ and $c_2(t)$ vary continuously with time, and at some sampled times they pass close to the fixed grid values.

We first demonstrate that probabilistic exact-stability guarantees are achievable for this dynamic point cloud. Consider Gaussian noise with standard deviation $\sigma = 0.002$. We verify the activation condition~\eqref{eq:taustarzero}:
\begin{equation}
\Gamma_{\mathrm{grid}} = 0.032 > \sqrt{2}\sigma\sqrt{d} = \sqrt{2}(0.002)\sqrt{2} = 0.004,
\end{equation}
which holds with a factor of 8 margin. Computing $\tau^\star$:
\begin{equation}
\tau^\star = \frac{0.032}{\sqrt{2}(0.002)} - \sqrt{2} = 11.3 - 1.41 \approx 9.9.
\end{equation}

With these parameters, the probability bound becomes:
\begin{equation}
\begin{aligned}
\mathbb{P}[\mathcal{E}] &\leq \min\left\{1, \binom{5}{2} \cdot 51 \cdot \exp\left[-\frac{(9.9)^2}{2}\right]\right\} \\
&= \min\left\{1, 510 \cdot \exp[-49.0]\right\} \\
&\approx \min\left\{1, 510 \times 10^{-22}\right\} \\
&\approx 10^{-19}.
\end{aligned}
\end{equation}
This is an exceptionally strong guarantee: even accounting for 510 potential threshold-crossing events across all pairs and times, the probability of any change in the Crocker diagram is negligible.

Now consider what happens as we increase the noise level. With $\sigma = 0.008$:
\begin{equation}
\tau^\star = \frac{0.032}{\sqrt{2}(0.008)} - \sqrt{2} \approx 2.83 - 1.41 = 1.42,
\end{equation}
yielding:
\begin{equation}
\mathbb{P}[\mathcal{E}] \leq \min\left\{1, 510 \cdot \exp[-1.01]\right\} \approx \min\{1, 186\}.
\end{equation}
The bound exceeds 1 and becomes vacuous. This does not mean the Crocker diagram will necessarily change---the union bound is conservative, and in practice the diagram may remain stable. Rather, it means our theoretical framework can no longer guarantee stability at this noise level.

This example illustrates several key insights. First, probabilistic exact-stability guarantees \emph{are} achievable for dynamic point clouds when the noise is sufficiently small relative to the grid-threshold clearance. Second, for dynamic systems where critical distances vary continuously, $\Gamma_{\mathrm{grid}}$ can be substantially smaller than the gap $\Delta$ between critical-distance families, because the moving critical curves may pass close to fixed grid values. Third, the prefactor $\binom{m}{2} n_t$ grows with the number of time samples, so probabilistic guarantees become more demanding as temporal resolution increases---a tradeoff between tracking fidelity and stability certification.

\subsection{Scope Limitations for Large Point Clouds}
\label{sec:cellnoise}

In Section~\ref{sec:cellmigration}, we showed that exact deterministic stability is difficult to verify for large biological point clouds because the grid clearance may be extremely small or impractical to compute. The same limitation applies to probabilistic exact stability: the activation condition $\Gamma_{\mathrm{grid}} > \sqrt{2}\sigma\sqrt{d}$ requires knowing $\Gamma_{\mathrm{grid}}$, which in turn requires evaluating the distance from every pairwise distance to every grid threshold at every sampled time.

For epithelial cell sheets with $m \approx 500$ cells, there are $\binom{500}{2} \approx 125{,}000$ pairwise distances. In a generic disordered configuration, these distances are nearly all distinct, and many will lie close to at least one of the grid thresholds $\widehat{\varepsilon}_j$. Consequently, $\Gamma_{\mathrm{grid}}$ is likely to be extremely small---potentially smaller than any realistic measurement precision $\sigma$. When this occurs, the activation condition $\Gamma_{\mathrm{grid}} > \sqrt{2}\sigma\sqrt{d}$ cannot be satisfied, and the probabilistic bound~\eqref{eq:prob-stability-bound} becomes vacuous.

Even if we could somehow verify the activation condition, the prefactor $\binom{m}{2} n_t$ would be enormous. With $m = 500$ and $n_t = 51$:
\begin{equation}
\binom{500}{2} \cdot 51 \approx 6.4 \times 10^6.
\end{equation}
To overcome this prefactor and achieve a probability bound below 0.01, we would need 
\begin{equation}
\exp[-(\tau^\star)^2/2] < 1.6 \times 10^{-9},
\end{equation}
requiring $\tau^\star > 6.4$. For $d = 2$, using $\Gamma_{\mathrm{grid}} = \sqrt{2}\sigma(\sqrt{d} + \tau^\star)$, this means $\Gamma_{\mathrm{grid}} > \sqrt{2}\sigma(1.41 + 6.4) \approx 11.0\sigma$. With typical tracking uncertainty $\sigma \approx 0.7~\mu\text{m}$ (consistent with the deterministic bound $\delta \approx 2\sigma$ from Section~\ref{sec:cellmigration}), we would need $\Gamma_{\mathrm{grid}} > 7.7~\mu\text{m}$---far larger than the sub-micron clearances expected in dense, disordered point clouds.

This analysis reveals that probabilistic exact-stability guarantees face the same fundamental scope limitation as deterministic exact stability: both require the grid-threshold clearance to be verifiable and sufficiently large. For large generic point clouds, neither condition is typically satisfied.

The appropriate stability tool for such settings remains the bounded-change theorem (Theorem~\ref{thm:global_bound}), which makes no assumptions about the distance spectrum and bounds total Betti-number changes in terms of local density $\Lambda$. While this bound does not provide probabilistic guarantees of exact invariance, it explains why Crocker diagrams remain informative even when exact-stability conditions cannot be verified: the total change scales with local density rather than global point count, localizing the effects of perturbations.

In practice, for large biological datasets, we recommend: (1) using the bounded-change theorem to understand worst-case behavior; (2) empirically validating stability by computing Crocker diagrams for multiple noise realizations; and (3) focusing analysis on smaller scales where $\Lambda$ is smaller and topological features are more robust.

Our stability framework thus far has assumed that the point set maintains a fixed cardinality. We now turn to the more general case where the number of points itself changes over time.

\section{Stability to Point Churn}
\label{sec:point-churn}

Real-world dynamical systems may not maintain a fixed set of entities throughout their evolution. Objects may enter or exit a scene, particles may be created or annihilated, and adaptive computational methods may refine or coarsen their discretization. These phenomena---which we refer to as \emph{point churn}---present a fundamental challenge for topological data analysis: how do we maintain meaningful summaries when the underlying point cloud's cardinality changes over time?

This section analyzes the stability of Crocker diagrams under such dynamic changes to the point set. Two key quantities control the topological disruption:
\begin{itemize}
\item $q$, the number of points inserted or deleted in a single modification event, usually much less than the total number of points $m$, and
\item $\Lambda(\varepsilon)$, the local-density parameter from Definition~\ref{def:lambda}, i.e.\ the largest neighborhood size at scale $\varepsilon$
\end{itemize}

Our main finding is surprisingly reassuring: topological changes scale linearly with the amount of churn, independent of global system size. Specifically, inserting or deleting $q$ points can alter the $k$-th Betti number by at most $A_k(\Lambda) q$, where $A_k(\Lambda) = \binom{\Lambda}{k+1}$ depends only on local density. This linear scaling---far tighter than the worst-case $O(qm^{k+1})$ bound for arbitrary complexes---reflects the geometric constraints inherent in Vietoris--Rips filtrations.

We first prove the bound, then revisit the breathing pentagon example to show how Crocker diagrams maintain their descriptive power even when entities enter or leave the system.

\subsection{Theoretical Stability for Changing Point Sets}
\label{sec:point-churn-theory}

Having established stability results for fixed point sets under general and noisy perturbations, we now consider the case where points may be added or removed during the evolution of a dynamical system. To provide theoretical guarantees in this setting, we must characterize how topological features respond to changes in the underlying index set itself.

Let us formalize this scenario by partitioning the time domain into intervals where the index set remains constant:
\begin{equation}
0 = t_0 < t_1 < \dots < t_L,
\label{eq:time_intervals}
\end{equation}
with $I_\ell = [t_{\ell-1}, t_\ell)$ representing blocks where the active index set is constant. Within each interval $I_\ell$, our previous stability results apply directly. Define the grid-threshold clearance on interval $I_\ell$ as
\begin{equation}
\Gamma_\ell := \min_{t_i \in I_\ell} \min_{j=1,\ldots,n_\varepsilon} \min_{a<b} \bigl| \|p_a(t_i) - p_b(t_i)\| - \widehat{\varepsilon}_j \bigr|.
\end{equation}
If $2\delta < \Gamma_\ell$, then by Theorem~\ref{thm:exact_stability} the Crocker diagram agrees cell-by-cell within that interval. However, as discussed in Sections~\ref{sec:cellmigration} and~\ref{sec:cellnoise}, verifying such clearance conditions may be difficult for large or generic point clouds. In these settings, the bounded-change results below provide the appropriate stability guarantees.

As in Sections~\ref{sec:stability_theory}--\ref{sec:probabilistic}, stability should not be expected in near-threshold configurations. When many pairwise distances lie close to grid thresholds, inserting or deleting even a small number of points can trigger widespread edge creation or destruction, leading to substantial changes in Betti numbers. The bounded-change theorems quantify the maximum possible impact of such events, but do not preclude large observed changes in regimes of small clearance.

The challenge arises at time boundaries $t_\ell$ when the point index set changes, violating the assumptions of our previous stability theorems and potentially causing abrupt topological changes. To understand the magnitude of these changes, we analyze how Betti numbers respond to individual modification events.

Suppose a single modification event inserts or removes a subset of $q$ points from a simplicial complex built on $m$ existing points, where $q \ll m$. While dynamical systems may experience multiple such events over time, we analyze the stability impact of each individual modification; our time-partitioned framework then tracks the cumulative effects by applying these single-event bounds at each change point $t_\ell$. This reflects the typical case where each modification affects only a small fraction of the data. We use the standard convention that a $k$-simplex has $k+1$ vertices. We first bound the number of combinatorial changes; we then translate this to topological bounds. We begin with a worst-case analysis without imposing geometric constraints and bound how many simplices can appear or disappear:

\begin{lemma}[Simplex Count Bound]
\label{lem:simplex_count}
When a single modification event adds or removes $q$ points from a point cloud of $m$ points, the number of new or lost $k$-simplices is bounded by:
\begin{equation}
\text{Number of new or lost $k$-simplices} = O(qm^k) \text{ for } k \geq 0
\label{eq:k_simplices_bound}
\end{equation}
and likewise for $(k+1)$-simplices:
\begin{equation}
\text{Number of new or lost $(k+1)$-simplices} = O(qm^{k+1}) \text{ for } k \geq 0.
\label{eq:k_plus_1_simplices_bound}
\end{equation}
\end{lemma}

\begin{proof}
Let $\sigma \in \{0,1\}$ denote addition ($\sigma=0$) or deletion ($\sigma=1$) of points. This notation allows us to handle both cases uniformly: for additions we work with the original $m$ points, while for deletions we work with the remaining $m-q$ points. We will use the standard asymptotic relation $\binom{N}{s} = O(N^s)$ for fixed $s$ throughout. We treat the case $k=0$ separately from $k \geq 1$.

For $k=0$, a $0$-simplex is simply a vertex, so exactly $q$ vertices are created or destroyed:
\begin{equation}
\text{Number of new or lost $0$-simplices} = q = O(qm^0).
\end{equation}
For the companion bound with $k+1=1$ (edges), each modified vertex can participate in at most $m$ edges with existing vertices, giving
\begin{equation}
\text{Number of new or lost $1$-simplices} = O(qm) = O(qm^{0+1}).
\end{equation}

For $k \geq 1$, a new or lost $k$-simplex arises either entirely among the $q$ modified vertices or as a mixed simplex (i.e., a simplex that contains both modified and unmodified vertices). The first type contributes at most $\binom{q}{k+1} = O(q^{k+1})$ simplices by our asymptotic relation. For mixed simplices, we must count those with exactly $i$ modified vertices (where $1 \leq i \leq k$) and $k+1-i$ unmodified vertices:
\begin{equation}
\sum_{i=1}^{k} \binom{q}{i}\binom{m-\sigma q}{k+1-i}.
\label{eq:mixed_simplices}
\end{equation}
Using our asymptotic relation and the assumption that $q \ll m$, each term behaves as
\begin{equation}
\binom{q}{i}\binom{m-\sigma q}{k+1-i} = O(q^i (m-\sigma q)^{k+1-i}) = O(q^i m^{k+1-i}) = O(qm^k(q/m)^{i-1}),
\end{equation}
where the second equality uses $(m-\sigma q)^{k+1-i} \leq m^{k+1-i}$ since $\sigma q \leq q \ll m$. Since $q/m < 1$, successive terms in the sum decrease geometrically (the ratio of consecutive terms is $q/m$), and the sum is dominated by the $i=1$ term, yielding $O(qm^k)$ mixed simplices.

Combining both contributions, we have $O(q^{k+1}) + O(qm^k) = O(qm^k)$ new or lost $k$-simplices, where the mixed-simplex term dominates since $q^{k+1} \leq qm^k$ for $k \geq 1$ when $q \leq m$. The same reasoning with $k$ replaced by $k+1$ yields $O(qm^{k+1})$ for $(k+1)$-simplices.
\end{proof}

Having established bounds on the number of simplices affected by point insertions and deletions, we now translate these combinatorial results into topological consequences. Adding or removing a single simplex corresponds to adding or removing one column in the relevant boundary matrix ($\partial_k$ for a $k$-simplex, $\partial_{k+1}$ for a $(k+1)$-simplex). A rank-one column update changes the matrix rank by at most 1, so each simplex modification can alter $\beta_k$ by at most one. By combining this per-simplex bound with our simplex counts from Lemma~\ref{lem:simplex_count}, we can derive worst-case bounds on how dramatically the topology can change when points are added or removed from the system:

\begin{theorem}[Worst-Case Betti Number Change]
\label{thm:worst_case_betti}
When a single modification event adds or removes $q$ points from a point cloud of $m$ points, the maximum change in the $k$th Betti number is bounded by:
\begin{equation}
|\Delta \beta_k| = O(qm^{k+1}) \text{ for } k \geq 0.
\label{eq:worst}
\end{equation}
\end{theorem}

\begin{proof}
The total change in $\beta_k$ is bounded by the sum of the maximum changes from $k$-simplices and $(k+1)$-simplices. By Lemma \ref{lem:simplex_count}, we have:
\begin{equation}
|\Delta \beta_k| \leq O(q m^k) + O(q m^{k+1}) = O(q m^{k+1})
\end{equation}
Since $q m^k \leq q m^{k+1}$, this simplifies to $|\Delta \beta_k| = O(q m^{k+1})$.
\end{proof}

The implicit constants in this bound depend only on the dimension $k$; they are independent of $m$ and $q$. Expressing $q = p m$ where $0 < p < 1$ (the modified points as a proportion of the total), our bound becomes:
\begin{equation}
|\Delta \beta_k| = O(pm \cdot m^{k+1}) = O(p m^{k+2}) \text{ for } k \geq 0.
\label{eq:beta_k_proportion}
\end{equation}
Note that for $k=0$ (counting connected components), this gives the intentionally loose bound $O(pm^2)$ rather than the exact $O(pm)$ count of vertices. We accept this coarser bound for uniformity across all dimensions, though tighter control on $\beta_0$ can be retained if needed for specific applications.

This worst-case analysis assumes arbitrary simplicial complexes with no geometric constraints. Fortunately, when we consider the spatial proximity inherent in Vietoris--Rips complexes, this worst-case explosion is dramatically reduced to a linear effect controlled by local density rather than global size:

\begin{proposition}[Geometry-Aware Stability]
\label{prop:geom_aware}
For a Vietoris--Rips complex at scale $\widehat{\varepsilon}_j$, a single modification event affecting $q$ vertices can change the $k$th Betti number by at most:
\begin{equation}
|\Delta \beta_k| \leq A_k(\Lambda) \, q
\label{eq:delta_beta_k_local}
\end{equation}
where $A_k(\Lambda) = \binom{\Lambda}{k+1}$ and $\Lambda = \Lambda(\widehat{\varepsilon}_j)$ is the maximum neighborhood size at scale $\widehat{\varepsilon}_j$ over all vertices in both the pre- and post-modification point clouds.
\end{proposition}

\begin{proof}
In a Vietoris--Rips complex at scale $\widehat{\varepsilon}_j$, each vertex has at most $\Lambda(\widehat{\varepsilon}_j) - 1$ neighbors. Thus it participates in at most $\binom{\Lambda(\widehat{\varepsilon}_j)-1}{k}$ $k$-simplices and $\binom{\Lambda(\widehat{\varepsilon}_j)-1}{k+1}$ $(k+1)$-simplices. Modifying $q$ vertices can therefore affect at most $q \cdot \binom{\Lambda(\widehat{\varepsilon}_j)-1}{k}$ $k$-simplices and $q \cdot \binom{\Lambda(\widehat{\varepsilon}_j)-1}{k+1}$ $(k+1)$-simplices. Since each simplex modification changes $\beta_k$ by at most one, the total change is bounded by:
\begin{equation}
|\Delta \beta_k| \leq q \cdot \left[\binom{\Lambda(\widehat{\varepsilon}_j) - 1}{k} + \binom{\Lambda(\widehat{\varepsilon}_j) - 1}{k+1} \right].
\end{equation}
Applying Pascal's rule $\binom{n-1}{k} + \binom{n-1}{k+1} = \binom{n}{k+1}$ to the bracketed term gives $\binom{\Lambda(\widehat{\varepsilon}_j)}{k+1} = A_k(\Lambda)$, completing the proof.
\end{proof}

To quantify the effect across an entire Crocker diagram, we extend our analysis to all affected time columns. In the theorem below, the factor $(n_t - i_\ell + 1)$ represents the number of time frames occurring after a point-set modification event, where $n_t$ is the total number of time frames.

\begin{theorem}[Global Stability with Changing Point Sets]
\label{thm:global_changing}
Let $B_k^{\text{orig}}$ and $B_k^{\text{mod}}$ denote the $k$-th Betti number matrices for the original and modified point clouds. If the point set changes at time $t_\ell$ (corresponding to time index $i_\ell$), the total difference across the entire Crocker diagram under the $\ell_1$-norm (sum of absolute differences) is bounded by:
\begin{equation}
\| B_k^{\text{mod}} - B_k^{\text{orig}} \|_1
\leq
(n_t - i_\ell + 1) \cdot q \cdot
\sum_{j=1}^{n_\varepsilon}
A_k(\Lambda(\widehat{\varepsilon}_j)).
\label{eq:final_bound}
\end{equation}
\end{theorem}

\begin{proof}
For each scale $\widehat{\varepsilon}_j$, the total change in the $j$-th scale row of the Crocker matrix is:
\begin{equation}
\sum_{i:\,t_i \geq t_\ell}
|\beta_k^{\text{mod}}(\widehat{\varepsilon}_j, t_i) - \beta_k^{\text{orig}}(\widehat{\varepsilon}_j, t_i)|
\leq
(n_t - i_\ell + 1) \cdot q \cdot A_k(\Lambda(\widehat{\varepsilon}_j)),
\label{eq:row_sum}
\end{equation}
Summing over all scale rows gives the full $\ell_1$-norm difference:
\begin{equation}
\| B_k^{\text{mod}} - B_k^{\text{orig}} \|_1
=
\sum_{j=1}^{n_\varepsilon}
\sum_{i:\,t_i \geq t_\ell}
|\beta_k^{\text{mod}}(\widehat{\varepsilon}_j, t_i) - \beta_k^{\text{orig}}(\widehat{\varepsilon}_j, t_i)|.
\label{eq:l1_norm}
\end{equation}
Since the inequality in \eqref{eq:row_sum} holds for each $t_i \geq t_\ell$, summing preserves the bound. Substituting \eqref{eq:row_sum} into \eqref{eq:l1_norm} yields our final bound.
\end{proof}

In summary, the effect of point-set churn is determined by local density rather than global size. This geometry-aware stability bound has several important properties: it scales linearly in the number of inserted or deleted points $q$, with a constant factor depending only on local density $\Lambda(\widehat{\varepsilon}_j)$ and the homology dimension $k$. The bound applies across all filtration scales and is much tighter than the earlier worst-case bound \eqref{eq:worst}, which depends on the global point count $m$ rather than local density. 

For dynamical systems experiencing multiple modification events over time, the cumulative topological change is bounded by the sum of individual event bounds. Because events occur in disjoint time intervals and the $\ell_1$ norm obeys the triangle inequality, the global bound for a sequence of events is simply the sum of the per-event bounds.

\subsection{Analytical Application: Breathing Pentagon with Insertion}

To illustrate our bound on point insertions and deletions, we revisit the breathing pentagon from Section~\ref{sec:breathing-poly}. The example we will now study is a small one, deliberately engineered to demonstrate how local geometric constraints govern topological disruption, thereby providing a controlled test of our theoretical framework. Consider the moment $t^\star$ when the pentagon's radius equals $a(t^\star)=1$, making the side-length (first critical scale) $c_1(t^\star) = 2a(t^\star)\sin(\pi/5) \approx 1.176$. Recalling the setup from Section~\ref{sec:breathing-poly}, the pentagon vertices are located at:
\begin{equation}
\begin{alignedat}{3}
v_0 &= (1,0) & & \\
v_1 &= (\cos 72^\circ,\sin 72^\circ) &\quad &\approx (\phantom{-}0.309, \phantom{-}0.951)\\
v_2 &= (\cos 144^\circ,\sin 144^\circ) & &\approx (-0.809, \phantom{-}0.588)\\
v_3 &= (\cos 216^\circ,\sin 216^\circ) & &\approx (-0.809, -0.588)\\
v_4 &= (\cos 288^\circ,\sin 288^\circ) & &\approx (\phantom{-}0.309, -0.951)
\end{alignedat}
\end{equation}

At time $t^\star$, we insert a new point $v_\ast$ at the geometric midpoint between the non-adjacent vertices $v_0$ and $v_2$, giving $v_\ast = (0.096, 0.294)$. We analyze the resulting Vietoris--Rips filtration at this single time frame. If the inserted point persists for $r$ consecutive frames, the global bound in Theorem~\ref{thm:global_changing} scales by the factor $r$.

By this positioning, $v_\ast$ is equidistant from the non-adjacent vertices $v_0$ and $v_2$ at distance approximately $0.951$, while being closest to $v_1$ at distance $0.691$. The distances from $v_\ast$ to $v_3$ and $v_4$ are both $1.263$. Crucially, the distance between $v_0$ and $v_2$ themselves is much larger: $c_2(t^\star) = 2\sin(2\pi/5) \approx 1.902$.

With these distances established, we can trace the topological evolution as $\varepsilon$ increases, shown in Figure~\ref{fig:breathing-pentagon-cycle}. The key observation is that the insertion \emph{destroys} the pentagon's 1-cycle earlier than it would otherwise die, rather than creating a new cycle.

\begin{figure}[ht!]
\centering

% === Customizable parameters ===
\def\FIGscale{1.2}
\def\RadialShift{0.17}
\def\DiskOpacity{0.5}      % ε/2-ball opacity
\def\DiskColor{black!15}  % ε/2-ball fill color
\def\PointSize{1pt}        % Dot size
\def\LabelSize{\scriptsize}

% === Coordinate macro ===
\newcommand{\PentagonCoords}{%
  \coordinate (v0) at (1, 0);
  \coordinate (v1) at ({cos(72)}, {sin(72)});
  \coordinate (v2) at ({cos(144)}, {sin(144)});
  \coordinate (v3) at ({cos(216)}, {sin(216)});
  \coordinate (v4) at ({cos(288)}, {sin(288)});
  \coordinate (vn) at (0.096, 0.294);}

% === Labeling macro ===
\newcommand{\LabelAllPointsRadial}{%
  % correct: start at origin (0,0) and go past the vertex
  \foreach \v/\name in {v0/$v_0$, v1/$v_1$, v2/$v_2$, v3/$v_3$, v4/$v_4$}%
      \node[font=\scriptsize] at ($(0,0)!1+\RadialShift!(\v)$) {\name};

  % label for v* (custom offset)
  \node[font=\scriptsize] at ($(vn)+(-0.12,0.12)$) {$v_{\ast}$};
}

% === Panel A ===
\begin{subfigure}[t]{0.42\textwidth}
\centering
\begin{tikzpicture}[scale=\FIGscale]
  \PentagonCoords
  \def\Rad{0.25} % ε = 0.5
  \foreach \P in {v0,v1,v2,v3,v4,vn}
    \fill[\DiskColor, opacity=\DiskOpacity] (\P) circle[radius=\Rad];
  \foreach \p in {v0,v1,v2,v3,v4,vn}
    \fill (\p) circle(\PointSize);
  \LabelAllPointsRadial
\end{tikzpicture}
\caption*{\small (A)\ $\varepsilon < 0.691$}
\end{subfigure}
%
% === Panel B ===
\begin{subfigure}[t]{0.42\textwidth}
\centering
\begin{tikzpicture}[scale=\FIGscale]
  \PentagonCoords
  \def\Rad{0.345}
  \foreach \P in {v0,v1,v2,v3,v4,vn}
    \fill[\DiskColor, opacity=\DiskOpacity] (\P) circle[radius=\Rad];
  \draw[thick] (vn)--(v1);
  \foreach \p in {v0,v1,v2,v3,v4,vn}
    \fill (\p) circle(\PointSize);
  \LabelAllPointsRadial
\end{tikzpicture}
\caption*{\small (B)\ $\varepsilon = 0.691$}
\end{subfigure}

\vspace{0.5em}

% === Panel C ===
\begin{subfigure}[t]{0.42\textwidth}
\centering
\begin{tikzpicture}[scale=\FIGscale]
  \PentagonCoords
  \def\Rad{0.475}
  \foreach \P in {v0,v1,v2,v3,v4,vn}
    \fill[\DiskColor, opacity=\DiskOpacity] (\P) circle[radius=\Rad];
  \foreach \x in {v0,v1,v2}
    \draw[thick] (vn)--(\x);
  \foreach \p in {v0,v1,v2,v3,v4,vn}
    \fill (\p) circle(\PointSize);
  \LabelAllPointsRadial
\end{tikzpicture}
\caption*{\small (C)\ $\varepsilon = 0.951$}
\end{subfigure}
%
% === Panel D ===
\begin{subfigure}[t]{0.42\textwidth}
\centering
\begin{tikzpicture}[scale=\FIGscale]
  \PentagonCoords
  \def\Rad{0.588}
  \foreach \P in {v0,v1,v2,v3,v4,vn}
    \fill[\DiskColor, opacity=\DiskOpacity] (\P) circle[radius=\Rad];
  \fill[blue!50!white, opacity=0.4] (vn)--(v0)--(v1)--cycle;
  \fill[blue!50!white, opacity=0.4] (vn)--(v1)--(v2)--cycle;
  \foreach \a/\b in {v0/v1, v1/v2, v2/v3, v3/v4, v4/v0}
    \draw[thick] (\a)--(\b);
  \foreach \x in {v0,v1,v2}
    \draw[thick,blue!70] (vn)--(\x);
  \foreach \p in {v0,v1,v2,v3,v4,vn}
    \fill (\p) circle(\PointSize);
  \LabelAllPointsRadial
\end{tikzpicture}
\caption*{\small (D)\ $\varepsilon = c_1 \approx 1.176$}
\end{subfigure}

\vspace{0.5em}

% === Panel E ===
\begin{subfigure}[t]{0.42\textwidth}
\centering
\begin{tikzpicture}[scale=\FIGscale]
  \PentagonCoords
  \def\Rad{0.631}
  \foreach \P in {v0,v1,v2,v3,v4,vn}
    \fill[\DiskColor, opacity=\DiskOpacity] (\P) circle[radius=\Rad];
  \fill[blue!50!white, opacity=0.4] (vn)--(v0)--(v1)--cycle;
  \fill[blue!50!white, opacity=0.4] (vn)--(v1)--(v2)--cycle;
  \fill[blue!50!white, opacity=0.4] (vn)--(v2)--(v3)--cycle;
  \fill[blue!50!white, opacity=0.4] (vn)--(v3)--(v4)--cycle;
  \fill[blue!50!white, opacity=0.4] (vn)--(v4)--(v0)--cycle;
  \foreach \a/\b in {v0/v1, v1/v2, v2/v3, v3/v4, v4/v0}
    \draw[thick] (\a)--(\b);
  \foreach \x in {v0,v1,v2,v3,v4}
    \draw[thick,blue!70] (vn)--(\x);
\foreach \p in {v0,v1,v2,v3,v4,vn}
    \fill (\p) circle(\PointSize);
  \LabelAllPointsRadial
\end{tikzpicture}
\caption*{\small (E)\ $\varepsilon = 1.263$}
\end{subfigure}
%
% === Panel F ===
\begin{subfigure}[t]{0.42\textwidth}
\centering
\begin{tikzpicture}[scale=\FIGscale]
  \PentagonCoords
  \def\Rad{0.951}
  \foreach \P in {v0,v1,v2,v3,v4,vn}
    \fill[\DiskColor, opacity=\DiskOpacity] (\P) circle[radius=\Rad];
  \fill[blue!50!white, opacity=0.4] (vn)--(v0)--(v1)--cycle;
  \fill[blue!50!white, opacity=0.4] (vn)--(v1)--(v2)--cycle;
  \fill[blue!50!white, opacity=0.4] (vn)--(v2)--(v3)--cycle;
  \fill[blue!50!white, opacity=0.4] (vn)--(v3)--(v4)--cycle;
  \fill[blue!50!white, opacity=0.4] (vn)--(v4)--(v0)--cycle;
  \foreach \a/\b in {v0/v1, v1/v2, v2/v3, v3/v4, v4/v0, v0/v2}
    \draw[thick] (\a)--(\b);
  \foreach \x in {v0,v1,v2,v3,v4}
    \draw[thick,blue!70] (vn)--(\x);
  \foreach \p in {v0,v1,v2,v3,v4,vn}
    \fill (\p) circle(\PointSize);
  \LabelAllPointsRadial
\end{tikzpicture}
\caption*{\small (F)\ $\varepsilon = c_2 \approx 1.902$}
\end{subfigure}

\caption{Filtration of the Vietoris--Rips complex for a regular pentagon with unit circumradius at time $t^\star$ (vertices $v_0, v_1, v_2, v_3, v_4$) after inserting point $v_\ast = (0.096, 0.294)$ at the geometric midpoint between non-adjacent vertices $v_0$ and $v_2$. Each panel shows the complex at threshold $\varepsilon$ using $\varepsilon/2$-balls (gray) around vertices; edges connect vertices at distance $\le \varepsilon$. (A)--(C)~At small scales, $v_\ast$ first forms an isolated component, then connects to $v_1$, then to $v_0$ and $v_2$. (D)~At $\varepsilon = c_1 \approx 1.176$, the pentagon side edges appear; since $v_\ast v_0$, $v_\ast v_1$, and $v_\ast v_2$ are already present from smaller scales, the flag property implies the triangles $(v_0,v_1,v_\ast)$ and $(v_1,v_2,v_\ast)$ (blue) appear immediately. (E)~At $\varepsilon = 1.263$, edges $v_\ast$--$v_3$ and $v_\ast$--$v_4$ appear, making $v_\ast$ adjacent to all pentagon vertices. The resulting fan of five triangles \emph{fills} the pentagon cycle, so $\beta_1 = 0$. (F)~Without the insertion, the pentagon cycle would persist until $\varepsilon = c_2 \approx 1.902$. Thus the insertion causes $\Delta\beta_1 = -1$ on $[1.263, 1.902)$: it destroys the cycle early rather than creating a new one. The observed $|\Delta\beta_1| = 1$ lies within our theoretical bound of 15 based on the local neighbor cap $\Lambda = 6$.}
\label{fig:breathing-pentagon-cycle}
\end{figure}
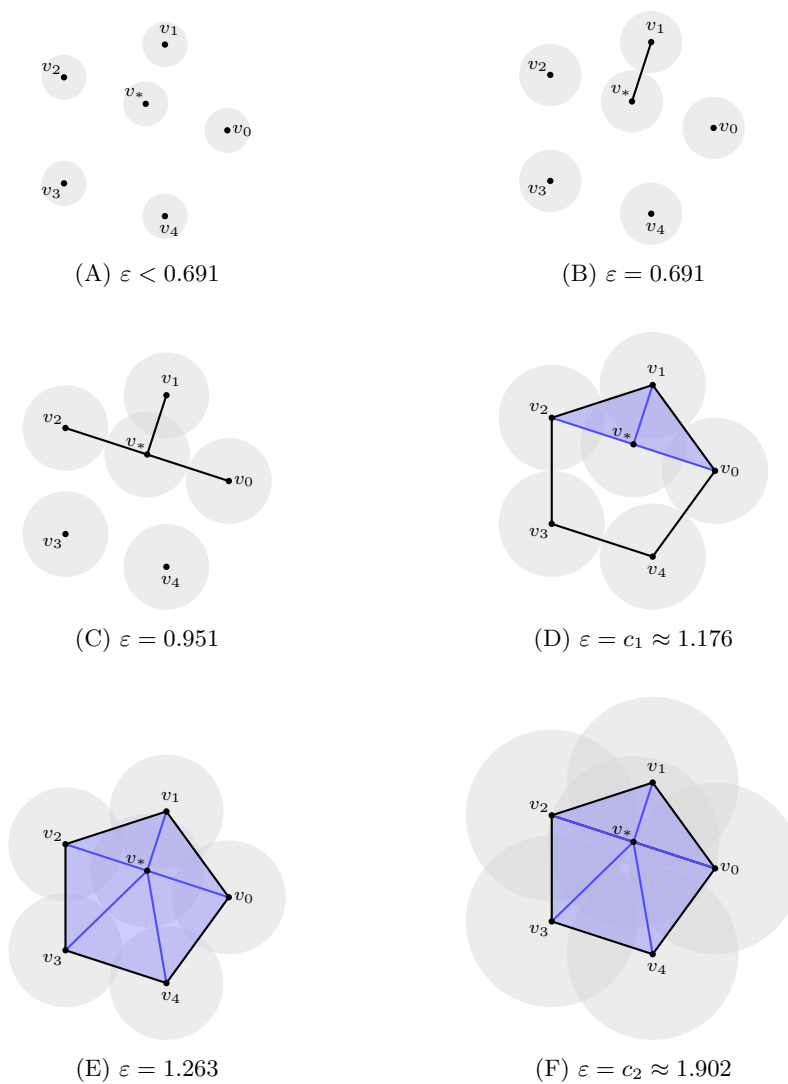

We now trace the filtration. For $\varepsilon < 0.691$, the inserted point $v_\ast$ forms its own connected component, so $\beta_0$ increases by 1 compared to the pentagon alone. When $\varepsilon$ reaches $0.691$, the edge $v_\ast$--$v_1$ appears, merging $v_\ast$ into the pentagon's component. At $\varepsilon = 0.951$, edges $v_\ast$--$v_0$ and $v_\ast$--$v_2$ appear simultaneously, connecting the new point to three pentagon vertices. Note that at this scale, the pentagon side edges (of length $c_1 \approx 1.176$) are not yet present.

When $\varepsilon$ reaches $c_1 \approx 1.176$, all pentagon side edges appear, creating the familiar pentagon cycle. At this moment, since the edges $v_\ast v_0$, $v_\ast v_1$, and $v_\ast v_2$ are already present (from smaller $\varepsilon$), the flag property implies that the triangles $(v_0, v_1, v_\ast)$ and $(v_1, v_2, v_\ast)$ appear immediately---all their edges now exist. Without the inserted point, we would have $\beta_1 = 1$. With $v_\ast$ present, the two triangles partially fill the cycle, but $\beta_1 = 1$ in both cases at this scale, so $\Delta\beta_1 = 0$.

The critical transition occurs at $\varepsilon = 1.263$, when edges $v_\ast$--$v_3$ and $v_\ast$--$v_4$ appear. Now $v_\ast$ is connected to \emph{all five} pentagon vertices. In a flag complex, this creates a ``fan'' of five triangles: $(v_0, v_1, v_\ast)$, $(v_1, v_2, v_\ast)$, $(v_2, v_3, v_\ast)$, $(v_3, v_4, v_\ast)$, and $(v_4, v_0, v_\ast)$. These five triangles together bound the pentagon cycle $v_0 \to v_1 \to v_2 \to v_3 \to v_4 \to v_0$, killing $\beta_1$. With $v_\ast$ present, $\beta_1 = 0$ for $\varepsilon \ge 1.263$ because the fan fills the cycle. Without $v_\ast$, $\beta_1 = 1$ would persist until $\varepsilon = c_2 \approx 1.902$, when diagonal $v_0$--$v_2$ appears and creates filling triangles. Thus $\Delta\beta_1 = -1$ on the interval $[1.263, 1.902)$: the insertion \emph{decreases} $\beta_1$ by causing the pentagon's cycle to be filled earlier.

We now validate the churn bound. The local neighbor count $\Lambda(\varepsilon)$ around $v_\ast$ evolves as:
\begin{equation}
\Lambda(\varepsilon)=
\begin{cases}
1 \quad \text{(only $v_\ast$ itself)},& \varepsilon < 0.691,\\
2 \quad \text{($v_\ast$ and $v_1$)},& 0.691 \leq \varepsilon < 0.951,\\
4 \quad \text{($v_\ast$, $v_0$, $v_1$, and $v_2$)},& 0.951 \leq \varepsilon < 1.263,\\
6 \quad \text{(all vertices)},& \varepsilon \geq 1.263.
\end{cases}
\end{equation}
Note that $v_3$ and $v_4$ enter the neighborhood simultaneously at $\varepsilon = 1.263$ due to symmetry, so $\Lambda$ jumps directly from 4 to 6.

With $q=1$ modified vertex, our local bound from Proposition~\ref{prop:geom_aware} uses the maximum value $\Lambda = 6$:
\begin{align}
|\Delta\beta_0| &\leq A_0(6) = \binom{6}{1} = 6,\\
|\Delta\beta_1| &\leq A_1(6) = \binom{6}{2} = 15.
\end{align}
For $k \geq 2$, we have $|\Delta\beta_k| = 0$ because six points in $\mathbb{R}^2$ cannot support homology above dimension one. The observed changes---$|\Delta\beta_0| = 1$ at small scales and $|\Delta\beta_1| = 1$ on $[1.263, 1.902)$---lie comfortably within these bounds.

By Theorem~\ref{thm:global_changing}, the bound scales linearly with both the number of modified points $q$ and the number of affected time frames, while remaining independent of the global point count. This example illustrates a general principle: inserting a point that connects to many existing vertices tends to ``cone off'' existing cycles, reducing $\beta_1$ rather than increasing it. Once $v_\ast$ connects to all pentagon vertices, it provides a filling disk for the cycle.

The behavior we observe validates our theory in two ways. First, even when we engineer a configuration for maximal local connectivity, the changes remain well within the predicted limits. Second, the construction demonstrates that topological disruption from insertions is governed by local geometry (the neighbor count $\Lambda$) rather than global point count---if the same pentagon were embedded in a much larger cloud, the local $\Lambda$ around $v_\ast$ would be unchanged, so the same finite bounds would apply.

\section{Conclusion}

Crocker diagrams have become a widely used tool for visualizing time-evolving topology. Until now, however, their empirically observed effectiveness lacked a firm theoretical basis. This paper addresses that gap by providing a unified framework of deterministic, probabilistic, and insertion--deletion bounds on Betti counts for time-varying point clouds constructed via a Vietoris--Rips filtration in Euclidean space.

Our analysis reveals a two-tier stability structure. For point clouds with structured pairwise distances---such as the breathing polygon model we develop---small deterministic perturbations leave a Crocker diagram \emph{exactly} unchanged whenever every point moves by less than $\Gamma/2$, where the grid-threshold clearance $\Gamma$ measures the minimum separation between grid thresholds and critical distances. For large or generic point clouds where pairwise distances are dense, this clearance may be extremely small or impractical to verify; in such settings, our bounded-change theorems provide the appropriate guarantees: Betti-number changes are controlled by local sampling density $\Lambda(\varepsilon)$ rather than global point-cloud size, explaining why Crocker diagrams remain informative even when exact-stability conditions cannot be certified.

The probabilistic theory extends this picture. Under isotropic Gaussian noise, the probability of any topological change decays exponentially once clearance exceeds the typical noise magnitude, $\Gamma_{\mathrm{grid}} > \sqrt{2}\sigma\sqrt{d}$. However, the union-bound prefactor $\binom{m}{2} n_t$ grows rapidly with the number of points and time frames, so meaningful probabilistic guarantees require either small point clouds (where the prefactor is manageable) or clearances substantially larger than the noise scale. When clearance is small relative to noise---as occurs in near-threshold configurations where many pairwise distances cluster around grid values---probabilistic exact-stability guarantees become vacuous, and one must rely on the deterministic bounded-change theorems to quantify the maximum possible impact of perturbations.

Finally, point insertions or deletions affect Betti counts at most linearly in the number of modified points, as only simplices incident to those points can change. This bound depends on local density $\Lambda$ rather than global size, supporting a locality principle: topological risk is governed by nearby geometry.

To illustrate these bounds in practice, we developed an analytically solvable breathing polygon model that revealed a non-monotonic relationship between side count and stability, with optimal stability at $m=5$ sides. The breathing pentagon also demonstrated the probabilistic framework: with sufficiently small noise, exact-stability guarantees are achievable even for dynamic point clouds, while larger noise renders the union bound vacuous. For epithelial cell imaging, we showed that exact stability---whether deterministic or probabilistic---is difficult to certify for large, disordered point clouds because the grid-threshold clearance is likely to be extremely small. The bounded-change theorem provides the appropriate theoretical tool for such applications, explaining empirical robustness without requiring clearance verification. The breathing pentagon with point insertion validated our insertion--deletion bounds by tracking a genuine homological event through a controlled scenario.

These results rest on three key assumptions: (i) we employ a Vietoris--Rips flag complex on points in~$\mathbb{R}^d$; (ii) for exact deterministic stability, each sampled scale must lie strictly between successive critical distances at all sampled times (the in-gap condition)---though our bounded-change and insertion--deletion results do not require this; and (iii) when randomness is considered, we assume independent, isotropic Gaussian perturbations. Relaxing any of these conditions will require new mathematical machinery beyond our current framework.

Our findings have immediate practical implications. For structured point clouds where pairwise distances are analyzable (e.g., regular lattices, symmetric configurations), practitioners can compute or bound the grid-threshold clearance and verify exact-stability conditions. For large or generic point clouds, the bounded-change theorems provide worst-case guarantees that scale with local density, guiding expectations about robustness. The probabilistic bound in~\eqref{eq:prob-stability-bound} yields an explicit inequality relating permissible noise variance, clearance, and confidence level---but practitioners should verify that the union-bound prefactor does not render the guarantee vacuous for their application. For dynamic data with intermittent point changes, our linear insertion--deletion bound (Section~\ref{sec:point-churn-theory}) quantifies worst-case topological distortion, guiding tolerance settings for tracking algorithms and experimental design.

Looking ahead, several promising directions emerge. First, one could develop adaptive scale grids that adjust at each time step based on observed data geometry, optimizing the resolution--stability trade-off while maintaining verifiable clearance. Second, our probabilistic results could be extended to more realistic noise scenarios, such as directional (anisotropic) measurement errors or non-Gaussian (heavy-tailed) noise distributions. Third, tighter bounds for large biological datasets might exploit typical constraints---such as bounded cell density, spatial regularity, or localized turnover---to yield practical guarantees where our current union bounds are vacuous. Finally, systematic benchmarking of Crocker diagrams against more detailed but computationally intensive methods like vineyards would clarify runtime--accuracy trade-offs and help determine when each approach is preferable in practice.

By providing rigorous guarantees, clarifying their scope of applicability, and offering actionable diagnostics, we hope this work encourages wider and more confident use of Crocker diagrams across scientific disciplines---from developmental biology to materials science---that grapple with evolving geometric structure.

\section*{Acknowledgments}
I am grateful to Lori Ziegelmeier, whose research inspired this project. Henry Adams and Maria-Veronica Ciocanel provided invaluable feedback on this manuscript. I used ChatGPT (OpenAI; GPT-4.5, May 11–15, 2025; GPT-5.2, Jan. 20–26, 2026) for limited assistance with language-level editing, typically via prompts such as ``tighten the prose in this paragraph'' or ``suggest clearer phrasing.'' I also used ChatGPT to generate initial drafts of TikZ code for schematic figures based on detailed descriptions of the intended geometry and layout. All such figure code was reviewed, corrected, and, where necessary, modified by hand. The manuscript was additionally processed with Refine.ink (January 25, 2026) solely for consistency checks. All mathematical and scientific content was conceived, derived, and verified by me, and I take full responsibility for the correctness and integrity of the work.

\bibliographystyle{siamplain.bst}
\bibliography{bibliography.bib}

\end{document}